\documentclass[reqno]{amsart}
\usepackage{amsmath, amsthm, amssymb, amstext}

\usepackage{hyperref,xcolor}
\hypersetup{
 pdfborder={0 0 0},
 colorlinks,
}
\usepackage[left=2.5cm,right=2.5cm,top=3cm,bottom=3cm]{geometry}
\usepackage{caption}
\usepackage{subcaption}
\usepackage{placeins}
\usepackage{cleveref}

\usepackage{mathtools}

\usepackage[colorinlistoftodos]{todonotes}
\usepackage{enumitem}
\newtheorem{theorem}{Theorem}[section]
\newtheorem{lemma}[theorem]{Lemma}
\newtheorem{e-proposition}[theorem]{Proposition}
\newtheorem{corollary}[theorem]{Corollary}

\newtheorem{e-definition}[theorem]{Definition}
\newtheorem{remark}[theorem]{\it Remark\/}

\def\div{{\rm div}}

\DeclareMathOperator*{\divergenz}{div}              %

\newcommand{\Cavg}{\mathcal C_{\mathrm{avg}}}
\newcommand{\Cdel}{\mathcal{C}_\delta}
\newcommand{\Cavgb}{\widetilde{\mathcal C}_{\mathrm{avg}}}
\newcommand{\Cdelb}{\widetilde{\mathcal{C}}_\delta}

\newcommand{\A}{\mathcal{A}}
\newcommand{\N}{\mathbb{N}}

\newcommand{\R}{\mathbb{R}}

\newcommand{\bw}{\mathbf{w}}

\newcommand{\bu}{\mathbf{u}}

\newcommand{\bphi}{\boldsymbol{\phi}}
\newcommand{\bpsi}{\boldsymbol{\psi}}

\newcommand{\eps}{\varepsilon}

\newcommand{\Om}{\Omega}

\newcommand{\into}{\int_{\Omega}}

\renewcommand{\l}{\left}
\renewcommand{\r}{\right}
\numberwithin{theorem}{section}
\numberwithin{equation}{section}

\newcommand{\cA}{\ensuremath{\mathcal{A}}}

\newcommand{\cE}{\ensuremath{\mathcal{E}}}

\newcommand{\cH}{\ensuremath{\mathcal{H}}}

\newcommand{\cT}{\ensuremath{\mathcal{T}}}

\newcommand{\bN}{\ensuremath{\mathbb{N}}}

\newcommand{\bR}{\ensuremath{\mathbb{R}}}

\newcommand{\bbu}{\boldsymbol{u}}
\newcommand{\bbv}{\boldsymbol{v}}

\newcommand{\bbw}{\boldsymbol{w}}

\newcommand{\bbU}{\boldsymbol{U}}
\newcommand{\bbV}{\boldsymbol{V}}
\newcommand{\bbJ}{\boldsymbol{J}}
\newcommand{\bbW}{\boldsymbol{W}}

\newcommand{\bbmu}{\boldsymbol{\mu}}
\newcommand{\bbnu}{\boldsymbol{\nu}}

\def\1{\boldsymbol{1}}

\title[Stationary solutions Cross-Diffusion-Cahn-Hilliard system]{Stationary solutions and large time asymptotics to a  Cross-Diffusion-Cahn-Hilliard system}

\author[J.\,Cauvin-Vila]{Jean Cauvin-Vila}
\address[J.\,Cauvin-Vila]{CERMICS, Ecole des Ponts, INRIA, MATHERIALS team-project, 6 et 8 av. Blaise Pascal, Cit\'e Descartes 77455 Marne-La-Vall\'ee, France}
\email{jean.cauvin-vila@enpc.fr}

\author[V.\,Ehrlacher]{Virginie Ehrlacher}
\address[V.\,Ehrlacher]{CERMICS, Ecole des Ponts, INRIA, MATHERIALS team-project, 6 et 8 av. Blaise Pascal, Cit\'e Descartes 77455 Marne-La-Vall\'ee, France}
\email{virginie.ehrlacher@enpc.fr}

\author[G.\,Marino]{Greta Marino}
\address[G.\,Marino]{Universit\"{a}t Augsburg, Institut f\"ur Mathematik, Universit\"{a}tsstra\ss e 12a, 86159 Augsburg, Germany}
\email{greta.marino@uni-a.de}

\author[J.-F.\,Pietschmann]{Jan-Frederik Pietschmann}
\address[J.-F.\,Pietschmann]{Universit\"{a}t Augsburg, Institut f\"ur Mathematik, Universit\"{a}tsstra\ss e 12a, 86159 Augsburg, Germany}
\email{jan-f.pietschmann@uni-a.de}



\begin{document}

\begin{abstract}
We study some properties of a multi-species degenerate Ginzburg-Landau energy and its relation to a cross-diffusion Cahn-Hilliard system. The model is motivated by multicomponent mixtures where cross-diffusion effects between the different species are taken into account, and where only one species does separate from the others. Using a comparison argument, we obtain strict bounds on the minimizers from which we can derive first-order optimality conditions, revealing a link with the single-species energy, and providing enough regularity to qualify the minimizers as stationary solutions of the evolution system. We also discuss convexity properties of the energy as well as long time asymptotics of the time-dependent problem. Lastly, we introduce a structure-preserving finite volume scheme for the time-dependent problem and present several numerical experiments in one and two spatial dimensions.
\end{abstract}

\maketitle


\section{Introduction}
The aim of this work is the study of a multi-species degenerate Ginzburg-Landau energy and its relation to a system of cross-diffusion Cahn-Hilliard equations which was recently studied in \cite{EMP2021}. The latter model describes the evolution of a multicomponent mixture where cross-diffusion effects between the different species are taken into account, and where only one species does separate from the others. This is motivated by multiphase systems where miscible entities may coexist in one single phase, see \cite{klinkert2015comprehension,Wenisch2016} for examples. Within this phase, cross-diffusion between the different species is taken into account in order to correctly account for finite size effects that may occur at high concentrations.\\
We assume that the mixture occupies an open, smooth and bounded domain $\Omega \subset \mathbb{R}^d$ with $d=1,2,3$ and that there are $n+1$ species in the mixture. We denote by $u_i(x, t)$, $i=0,\ldots, n$, the volumic fraction of the $i^{th}$ species at point $x\in \Omega$ and time $t\geq 0$ and set $\bbu=(u_0,\ldots, u_n)$. The dynamics of the system is governed by the free energy functional
\begin{equation}
\label{eq:energy}
\begin{aligned}
E(\bbu)&:= 
\int_\Omega \left[\sum_{i=0}^n (u_i \ln u_i -u_i +1) + \frac{\eps}{2} |\nabla u_0|^2 + \beta u_0 (1-u_0)\right] \;dx, 
\end{aligned}
\end{equation}
where $\eps$ and $\beta$ are positive constants.
Denoting by $\bbmu = D_{\bbu} E(\bbu)$ the chemical potential, the corresponding evolution system formally reads as
\begin{equation}\label{eq:system_intro}
 \begin{split}
  \partial_t  \bbu&=\div\left( M(\bbu) \nabla  \bbmu\right) \quad \text{ in } \Omega \times (0,+\infty),
 \end{split}
\end{equation}
where $M: \mathbb{R}^{n+1} \to \mathbb{R}^{(n+1) \times (n+1)}$ is a degenerate mobility matrix. More precisely, for every  $i \neq j= 0, \dots, n$, let $K_{ij}$ be  positive real numbers satisfying $K_{ij}= K_{ji}$, then for $\bbu \in \mathbb{R}^{n+1}_+$, it has entries
	\begin{align}
    \label{eq:mobility}
	\begin{aligned}
	M_{ij}(\bbu)&:= - K_{ij}u_i u_j \qquad &&\text{for all }  i\neq j= 0, \dots, n, \\
	M_{ii}(\bbu)&:= \sum_{0\leq k \neq i \leq n} K_{ik}u_iu_k && \text{for all }  i= 0, \dots, n.
	\end{aligned}
	\end{align}
As expected, due to their interpretation as volumic fractions, the quantities $u_i$ must satisfy
	\begin{equation}
	\label{eq:constraint}
	  0\leq u_i(x, t) \leq 1 \quad \text{for all $i=0,\dots,n$ \; and} \quad \sum_{i=0}^n u_i(x, t) = 1 \text{ for a.e. }x\in \Omega, \; t\in (0,+\infty),
	\end{equation}
and the constraint on the sum is referred to as the volume-filling constraint. The evolution system is supplemented with no-flux boundary conditions as well as initial conditions consistent with the constraints. The main result from \cite{EMP2021} is the existence of a solution to a suitable weak formulation of this problem.\\
The aim of this paper is twofold. First, we study some solutions to the stationary problem 
\begin{align}\label{eq:stationary}
    0=\div\left( M(\bbu) \nabla  \bbmu\right)\quad \text{ in } \Omega.
\end{align}
In general, the analysis of this system of coupled, degenerate elliptic equations is by no means straight forward. In this work, motivated by the gradient flow structure of the time-dependent equation highlighted above, we focus our study on the set of local minimizers of the energy functional \eqref{eq:energy}. The latter are natural candidates for solutions to \eqref{eq:stationary}, in the sense that one naturally expects that solutions of the time-dependent system should converge in the long time limit to one of these local minimizers. We acknowledge here that other stationary solutions may exist, but stress on the fact that local energy minimizers are of particular physical relevance for the present system. When the parameters are chosen such that the energy functional is convex, the unique minimizers are constants and we show that solutions to the evolution problem \eqref{eq:system_intro} converge to them exponentially fast. 

In the non-convex case, the dynamics is much more complex which leads us to the second aim of the paper: we introduce a finite volume scheme that preserves the structure of the continuous time-dependent system. The simulations demonstrate the capability of the scheme and allow to explore the dynamics for arbitrary parameter regimes. 

\medskip

Let us briefly review previous contributions on the respective components of our model.
\subsubsection*{Cross-diffusion systems with size exclusion} 
Systems of partial differential equations with cross-diffusion have gained a lot of interest in recent years \cite{Kfner1996InvariantRF,Chen2004,Chen2006,Lepoutre2012,Juengel2015boundedness} 
and appear in many applications, for instance the modelling of population dynamics of multiple species~\cite{Burger2016} or cell sorting as well as chemotaxis-like applications~\cite{Painter2002,Painter2009}.
\subsubsection*{Ginzburg-Landau Energy} 
In the case $n=1$, which implies $u_0 = 1- u_1$, \eqref{eq:energy} reduces to the classical Ginzburg-Landau energy with singular potential as introduced in \cite{CahnHilliard1958}. The works \cite{Gelantalis2014,Gelantalis2017} study the structure of energy minimizers to the functional 
\begin{align*}
     E_{\rm GL}(v) = \int_\Omega \frac{1}{2}|\nabla v|^2 + \frac{1}{4}(1-v^2)^2\;dx,
\end{align*}
when the system size is large and the mean value of the phase parameter $v$ is close to $-1$. This is motivated by nucleation phenomena, see \cite{Loganayaki2011_growth,Li2013_Nucleation}, which correspond to non-constant minimizers. The authors study the case when constant stationary states are local but not global minimizers and estimate the size of the energy barrier, i.e. the difference of the energy at the respective states. In particular, the authors prove bounds on the minimizers using suitable competitors which inspired part of the construction in the proof of Theorem~\ref{adm-min}.
\subsubsection*{Cahn-Hilliard equation} The scalar Cahn-Hilliard equation with constant mobility was introduced in \cite{CahnHilliard1958} as a model for phase separation. It is indeed the $H^{-1}$-gradient flow to \eqref{eq:energy} for two species. 
Existence of weak solutions was first shown in e.g. \cite{Elliott1986,Caffarelli1995} in the case of constant mobility, and later extended to degenerate, concentration dependent mobilities \cite{EG}. Regarding the long-time behaviour, for a constant mobility and in one spatial dimension, the authors in \cite{Otto2014} show that for initial data with bounded distance to a so-called kink state, algebraic convergence to equilibrium holds. This was further improved in \cite{Otto2019}. We also refer to \cite{abels2007,schimperna2007,schimperna2013} for long-time analysis in the case of logarithmic nonlinearity. More details can be found in the review \cite{novick2008cahn} and the monograph \cite{Miranville2019}.

Multi-species Cahn-Hilliard systems have been studied in several earlier works and usually consider an energy functional of the form 
\begin{align} \label{eq:elliot_energy}
E({\bf u}):= \int_\Omega\Big[ \Psi({\bf u}) + \frac{1}{2}\nabla {\bf u} \cdot \Gamma \nabla {\bf u}\Big] dx,
\end{align}
for some symmetric positive semi-definite  matrix $\Gamma \in \mathbb{R}^{(n+1) \times (n+1)}$ and bulk free-energy functional $\Psi$. 
In~\cite{elliott1991generalized}, Elliott and Luckhaus proved a global existence result for such a multiphase Cahn-Hilliard system with constant mobility and 
$\Gamma = \gamma {\rm I}$ for some $\gamma >0$. In~\cite{Elliot1997}, the authors generalized their result to the case of a degenerate concentration-dependent mobility matrix with a positive definite matrix $\Gamma$ while \cite{Brunk_CH_AC} study a system of Cahn-Hilliard/Allen-Cahn equations with cross-kinetic coupling. 
Recently, in~\cite{boyer2014hierarchy}, the authors proposed a novel hierarchy of multi-species Cahn-Hilliard systems which are consistent with the standard two-species Cahn-Hilliard system, and which read as the model introduced above with 
$\Gamma$  positive definite, a particular $\Psi$ and for a constant mobility matrix. Numerical methods for such systems were proposed and analyzed in several contributions, see e.g.~\cite{eyre1998unconditionally,Wu2017,chen2019a,bailo2021}. \medskip\\
Concerning coupled cross-diffusion Cahn-Hilliard systems, other than \cite{EMP2021}, the only work we are aware of is \cite{Huo2022ExistenceAW}, which treats the case where all species aim to separate, i.e. the case when $\Gamma$ in \eqref{eq:elliot_energy} is positive definite. 

\subsubsection*{Structure-preserving finite volume schemes.} The finite volume method is a classical discretization method to approximate conservation laws, see a pedagogical introduction in \cite{eymard2000finite}. It is a natural physical requirement for a discretization method to preserve as much as possible of the structure of the continuous problem such as conservation laws, nonnegativity or dissipation. In addition, such properties can be mathematically useful, since they enable to ``transfer" the mathematical analysis to the discrete level. A method that preserves the dissipation of an energy (resp. entropy) is often called ``energy-stable'' (resp. entropy-stable). Following the success of the entropy method, there has been considerable effort in order to preserve the entropic structure of scalar parabolic equations \cite{bessemoulin-chatard2012,chainais-hillairet2014,cances2017} and parabolic systems \cite{cances2019a,cances2020,cances2020d,carrillo2020,jungel2020,zurek2020,daus2021,jungel2023,cances2022a,herda2022,jungel2023a,cances2023c} at the discrete level. See also a review of energy-stable schemes for the Cahn-Hilliard equation in \cite{brachet2022}. 
\subsection{Contributions and outline}
Our work makes the following contributions

\begin{itemize}
\item Proving existence and uniform lower and upper bounds for the local
minimizers of \eqref{eq:energy} in the $L^\infty(\Omega)^{n+1}$ topology. We emphasise that the latter, in contrast to the results of \cite{EMP2021}, requires a construction which has to preserve not only the constraints \eqref{eq:constraint} but also the mass of the competitor, which significantly complicates the argument.
\item Gaining regularity of the minimizers from the Euler-Lagrange system, we show that they qualify as classical solutions to the stationary system. We also show that the Euler-Lagrange equation for the void species decouples, revealing a strong link with the single-species energy.
\item   We study the convexity properties of \eqref{eq:energy} and are able to give explicit quantitative bounds. In a particular parameters regime, we show that the minimizers are constant and that solutions to the dynamical system converge exponentially fast to them, for arbitrary initial data with finite energy. We give explicit rates of convergence. 
\item We introduce a two-point finite volume scheme that approximates the evolution problem \eqref{eq:system_intro}, preserving the constraints \eqref{eq:constraint}. The discrete free energy is shown to be nonincreasing, adapting the convex-concave splitting of \cite{EMP2021} to the discrete case. We provide numerical simulations to illustrate the behaviour of the scheme and to investigate the variety of stationary solutions in the long-time limit. 
\end{itemize}

\begin{remark}[Nonlocal and potential contributions to the energy]
We remark that most of the results of this work remain valid, after minor modifications, if potential or non-local interaction terms of the form 
$$
\int_\Omega V_i(x)u_i(x) \;dx \quad \text{ or }\quad c_{ij}\int_\Omega u_i L\ast u_j\;dx
$$
are added to the energy. Here, for all $0\leq i \leq n$, $V_i:\Omega \to \R$ is a given potential and $L:\Omega \to \R$ is an interaction kernel. All these functions must be sufficiently smooth. With these additions, existence of minimizers (\cref{existence-minim}), strict bounds (\cref{adm-min}) hold without any changes. First order optimality conditions have to be adapted and the regularity of solutions (\cref{thm:optimality}) is limited by the regularity of $L$ and $V:=(V_i)_{0\leq i \leq n}$. Under suitable assumptions on the matrix $C=(c_{ij})_{0\leq i,j\leq n}$ the numerical scheme can be adapted and still preserves the structure.
\end{remark}
The paper is organised as follows: Section \ref{sec:energy} contains an analysis of properties of the energy functional and establishes the link with stationary solutions. Section \ref{sec:conv-long} is dedicated to the large-time asymptotics in a globally stable regime. Section \ref{sec:scheme} is devoted to the introduction of a structure preserving finite volume scheme and some numerical results are presented in Section~\ref{sec:simulations}.

\subsection{The model}
We now present the system under consideration in full detail. For $\eps>0$ and $\beta >0$ we consider the energy functional given by \eqref{eq:energy}.
We define formally the chemical potentials as variational derivatives of the energy by
\begin{equation}
    \label{eq:mui}
	\mu_i := D_{u_i} E(\bbu) = \ln u_i 
	\qquad \text{for all }  i= 1, \dots, n,
\end{equation}
as well as
\begin{equation}
    \label{eq:mu0}
    \mu_0 := D_{u_0}E(\bbu) = \ln u_0 - \eps \Delta u_0 + \beta (1- 2u_0),
\end{equation}
so that $\bbmu:= (\mu_0, \mu_1, \dots, \mu_n) = D_{\bbu} E(\bbu)$. Furthermore, we introduce the auxiliary variables
\begin{equation}
    \label{eq:wi}
    w_i = \ln u_i - \ln u_0, ~ i=1,\dots,n,
\end{equation}
as well as
\begin{equation}
\label{eq:w0}
	w_0 := -\eps \Delta u_0 + \beta(1-2u_0).
\end{equation}
With these definitions, \eqref{eq:system_intro} can be rewritten as
\begin{equation}
\label{eqi}
\begin{aligned}
 \partial_t u_i &= \divergenz \l( \sum_{0 \leq j \neq i \leq n} K_{ij} u_i u_j \nabla (\mu_i-\mu_j) \r) \\ 
  &= \divergenz \l( \sum_{0 \leq j \neq i \leq n} K_{ij} u_i u_j \nabla (w_i-w_j) \r) \\
 &= \divergenz \l( \sum_{0 \leq j \neq i \leq n } K_{ij} ( u_j \nabla u_i - u_i \nabla u_j) - K_{i0} u_i u_0 \nabla w_0 \r),
\end{aligned}
 \end{equation}
 for $i=1,\dots,n$ and
 \begin{equation}
 \label{eq0}
     \begin{aligned}
         \partial_t u_0 &= \divergenz \l( \sum_{j=1}^n K_{0j} u_0 u_j \nabla (\mu_0-\mu_j) \r) \\
         &= \divergenz \l( \sum_{j=1}^n K_{0j} u_0 u_j \nabla (w_0 -w_j) \r) \\
         &= \divergenz \l( \sum_{j=1}^n K_{0j} (u_j \nabla u_0 - u_0 \nabla u_j + u_0 u_j \nabla w_0) \r).
     \end{aligned}
 \end{equation}
The conservative form \eqref{eq:system_intro} together with the zero-flux boundary conditions suggest that the mass of each species is conserved along the evolution. Therefore, given fixed masses $m_0,\dots,m_n >0$ such that $\displaystyle\sum_{j=0}^n m_j = |\Om|$, we will look for solutions to \eqref{eq:stationary} in the admissible set
\begin{equation*}
\begin{split}
	\cA_m:= \biggl\{\bbu:= (u_0, \dots, u_n) \in  (L^\infty(\Om))^{n+1}: \, u_i \ge 0, \, &\into u_i \;dx= m_i, \, i=0,\ldots, n,  \\
 \sum_{j=0}^n u_j=1 \; \text{a.e. in } \Om \; \text{and } u_0 \in H^1(\Om)\biggr\}.
\end{split}
\end{equation*}
Note that $\cA_m$ is non-empty, convex, and that for any $\bbu \in \mathcal{A}_m$, it holds $ 0 \le u_i \le 1$ for all $i=0, \dots, n$. 

\medskip

\section{Minimizers of the energy functional}
\label{sec:energy}

In this section we use the direct method of the calculus of variations to prove the existence of minimizers to the energy \eqref{eq:energy} over the set $\cA_m$: 
\begin{equation}
    \label{eq:min}
    \min_{\bbu \in \cA_m} E(\bbu).
\end{equation}
Arguing by means of competitors, we further obtain strict bounds which then allow for higher regularity by making use of the optimality conditions. In consequence, minimizers are solutions to the stationary problem \eqref{eq:stationary}. 

\begin{lemma}
\label{existence-minim}
Let  $E\colon \A_m \to \R$ be defined by \eqref{eq:energy}. Then, $E$ has at least one minimizer.
\end{lemma}

\begin{proof}
We apply the direct method of calculus of variations. First, using the non-negativity of the function $ [0, 1] \ni x \mapsto x \ln x- x+ 1$, together with the fact that $|\nabla v_0|^2 \ge 0$ and that $v_0 (1- v_0) \ge 0 $ for any $\bbv = (v_0, \ldots, v_n) \in \cA_m$, we obtain that $E$ is nonnegative on $\cA_m$. Moreover, it is clear that $\cA_m$ contains constant solutions with finite energy. 
Thus, there exists a minimizing sequence $\l(\bbv^{(p)}\r)_{p \in \N} \subset \A_m$ such that $E\l(\bbv^{(p)}\r)$ is bounded and
	\[
	\lim_{p \to \infty} E\bigl(\bbv^{(p)}\bigr)= \inf_{\A_m} E.
	\]
%
%
In particular, we  have that $\bigl(\bigl\|\nabla v_0^{(p)} \bigr\|_{L^2(\Om)} \bigr)_{p \in \N}$ is bounded as well. Therefore, without relabelling, up to the extraction of a subsequence, there exists $u_0\in H^1(\Omega)$ such that $\nabla v_0^{(p)} \rightharpoonup \nabla u_0 $ weakly in $L^2(\Om)$, and thanks to the uniform $L^{\infty}$-bound we have $v_0^{(p)} \to u_0$ strongly in $L^q(\Om)$ for every $1 \le q< \infty $ and a.e. in $\Om$. Furthermore, since $v_i^{(p)}$ is bounded in $L^2(\Om)$ by construction, it follows that, up to the extraction of a subsequence, for all $i=1,\dots,n$, there exists $u_i \in L^2(\Omega)$ such that $v_i^{(p)} \rightharpoonup u_i$ weakly in $L^2(\Om)$. We also easily obtain that $0\leq u_i \leq 1$ almost everywhere on $\Omega$. Then the convexity of the integrands together with the strong continuity of the functional (dominated convergence) imply the lower-semicontinuity 
	\[
	\into u_i\ln u_i \;dx \le \liminf_{p \to \infty} \into v_i^{(p)} \ln v_i^{(p)} \;dx
	\]
as well as 
	\begin{equation*}
	\into |\nabla u_0|^2 \;dx \leq  \liminf_{p\to \infty} \into \bigl|\nabla v_0^{(p)}\bigr|^2 \;dx.
	\end{equation*}
Furthermore, the weak convergence in $L^2(\Om)$ yields 
	\[
	\into \bigl(-v_i^{(p)}+ 1 \bigr) \;dx \to \into \bigl(-u_i+ 1\bigr) \;dx \quad \text{for all } i= 0, \dots, n,
	\]
while the strong convergence gives
	\[
	\into v_0^{(p)} \bigl(1- v_0^{(p)}\bigr) \;dx \to \into u_0 \l(1- u_0 \r) \;dx.
	\]
 This implies 
	\begin{equation*}
	E(\bbu) \leq \liminf_{p\to \infty} E (\bbv^{(p)}) = \inf_{\A_m} E,
	\end{equation*}
and that $\int_\Omega u_i\,dx = m_i$ for $i = 0,\ldots, n$. Finally, the weak convergence in $L^2(\Omega)$ also yields that $\displaystyle \sum_{i=0}^n u_i = 1$ almost everywhere in $\Omega$ so that $\bbu \in \mathcal{A}_m$.
The conclusion follows.
\end{proof}

\begin{remark}
We point out that the uniqueness of the minimizer is neither guaranteed nor expected, due to the non-convexity of the energy functional.
\end{remark}
A remarkable property is that the minimizers are in the interior of the set $\mathcal A_m$, i.e. they strictly satisfy the box constraints in \eqref{eq:constraint}. This is shown by constructing suitable competitors in the following theorem.

\begin{theorem}
\label{adm-min}
Let $E$ be the energy functional given by \eqref{eq:energy}. Then, there exists a constant $\delta>0$ such that for every local minimizer $\bbu \in \A_m$ of $E$ for the $L^\infty(\Omega)^{n+1}$ topology, it holds that
	\[
	\delta \le u_i \quad \text{a.e, for all }  i= 0, \dots, n,
	\]
which, together with the volume-filling constraint in \eqref{eq:constraint}, implies the upper bound 
\begin{align*}
    u_i \leq 1-n \delta  \quad \text{a.e, for all }  i=0,\dots,n.
\end{align*}
\end{theorem}

\begin{proof}
Let $\bbu:=(u_0, \ldots, u_n) \in \mathcal{A}_m$ be a local minimizer of $E$ for the $L^\infty(\Omega)^{n+1}$ topology, i.e. there exists $\epsilon>0$ such that for all $\bbv:=(v_0, \ldots, v_n) \in \cA_m$ with $\displaystyle\|\bbv - \bbu\|_{L^\infty(\Omega)}:= \mathop{\max}_{i=0,\ldots,n}\|v_i - u_i\|_{L^\infty(\Omega)}\leq \epsilon$, necessarily $E(\bbv) \geq E(\bbu)$ holds. In order to prove the assertion, we proceed as follows: first we show that there exists $\delta_0>0$ such that $\delta_0 \le u_0 \le 1-\delta_0$. Then we proceed to show that there exists  $\delta_i>0$ such that $ \delta_i \le u_i$ for $i= 1,\ldots,n$. Finally, $\delta := \min(\delta_0,  \delta_1, \dots, \delta_n)$ is the constant that appears in the statement.
\medskip\\
\underline{Step 1: $\delta_0 \le u_0$}. 
Arguing by contradiction, we assume that for all $\displaystyle \min\left(\epsilon, \frac{m_0}{2|\Om|}\right) \geq \delta>0$  the set
    \[
    \Cdel= \{x \in \Om : \, u_0(x)< \delta\}
    \]
is such that $|\Cdel|>0$. We further define the set
    \[
        \Cavg := \left\{ x \in \Omega \; : \; u_0(x) > \frac{m_0}{|\Omega|} \right\},
    \]
i.e. the part of $\Omega$ on which $u_0$ strictly exceeds its average. 
Note that 
    \[
    \mathcal C_{\frac{m_0}{2|\Omega|}} \cap \Cavg= \emptyset \quad \text{and} \quad \bigl| \mathcal C_{\frac{m_0}{2|\Omega|}} \bigl| >0 
    \]
  imply  $|\Cavg| >0$. 
We then define, for all $0 \leq \lambda \leq 1$, the perturbed function
    \[
    u_0^{\delta,\lambda}=
    \begin{cases}
        \delta & \text{in } \Cdel \\
        (1-\lambda)u_0+ \lambda \frac{m_0}{|\Om|} & \text{in } \Cavg \\
        u_0 & \text{otherwise},
    \end{cases}
    \]
which satisfies $0 \le u_0^{\delta,\lambda} \le 1$ and $u_0^{\delta,\lambda}\in H^1(\Omega)$. Moreover, for any $0\leq \lambda \leq \epsilon/2$, it holds that $\|u_0 - u_0^{\delta,\lambda}\|_{L^\infty(\Omega)}\leq \epsilon$. Observe that 
 \begin{align*}
     \int_\Omega u_0^{\delta,0}\; dx = \int_{\Cdel} \delta\;dx + \int_{\Omega \setminus \Cdel} u_0\;dx > \int_{\Cdel} u_0\;dx + \int_{\Omega \setminus \Cdel} u_0\;dx = m_0,
 \end{align*}
 while on the other hand 
 \begin{align*}
  \int_\Omega u_0^{\delta,\lambda}\,dx& = \int_{\Cdel } \delta\;dx + \int_{\Omega \setminus (\Cdel \cup \Cavg) } u_0\;dx + \lambda \int_{\Cavg}\frac{m_0}{|\Omega|}\;dx + (1-\lambda) \int_{\Cavg}u_0\;dx\\ 
     & = \delta |\Cdel| + \int_{\Omega \setminus \Cdel} u_0\;dx + \lambda \int_{\mathcal{C}_{\rm avg}}\left(\frac{m_0}{|\Omega|}-u_0\right)dx.\\
 \end{align*}
In order to estimate the last integral of the previous equality we observe that
 \begin{align*}
     0 &= \int_\Omega \left(\frac{m_0}{|\Omega|} - u_0\right) dx \\
     &= \int_{\mathcal C_{\frac{m_0}{2|\Omega|}}} \left(\frac{m_0}{|\Omega|} - u_0\right) dx + \int_{\Omega \setminus (\mathcal C_{\frac{m_0}{2|\Omega|}} \cup \Cavg)} \left(\frac{m_0}{|\Omega|} - u_0\right) dx + \int_{\Cavg} \left(\frac{m_0}{|\Omega|} - u_0\right)dx \\
     &\ge \frac{m_0}{2|\Omega|}|\mathcal C_{\frac{m_0}{2|\Omega|}}| + \int_{\Cavg}\left(\frac{m_0}{|\Omega|} - u_0\right)dx,
 \end{align*}
 which implies
 \begin{align*}
 \int_{\Cavg}\l(\frac{m_0}{|\Omega|}-u_0 \r)\;dx \le -\frac{m_0}{2|\Omega|}|\mathcal C_{\frac{m_0}{2|\Omega|}}|.
 \end{align*}
Then, from the previous calculations, we conclude that for any $0< \delta < \min\left( \epsilon, \frac{m_0}{2|\Omega|}\right)$ and any $\lambda\in [0,\epsilon/2]$, 
 \begin{align}\label{eq:lambda}
     \int_\Omega u_0^{\delta,\lambda}\,dx \le \l(\delta- \lambda \frac{m_0}{2|\Omega|}\r) |\mathcal C_{\frac{m_0}{2|\Omega|}}| + m_0.
 \end{align}
Let us now assume that $\delta$ is chosen so that $\delta < \frac{\epsilon m_0}{4|\Omega|}$. Then, it holds that
 \begin{align*}
     \int_\Omega u_0^{\delta,\epsilon/2}\,dx \le \l(\delta- \epsilon\frac{m_0}{4|\Omega|}\r) |\mathcal C_{\frac{m_0}{2|\Omega|}}| + m_0 < m_0.
 \end{align*}
Thus, for all $0< \delta < \frac{\epsilon m_0}{4|\Omega|}$, there exists $\lambda_\delta^*\in (0,\epsilon/2)$ such that the function $u_0^\delta := u_0^{\delta,\lambda^*_\delta}$ satisfies 
\begin{align}\label{eq:u0_perb_mass}
     \int_\Omega u_0^\delta \,dx = m_0.
 \end{align} 
Furthermore, it holds from (\ref{eq:lambda}) that $\lambda^*_\delta$ necessarily satisfies $\lambda_\delta^* \leq \frac{2|\Omega|\delta}{m_0}$, so that $\mathop{\lim}_{\delta \to 0} \lambda_\delta^* = 0$.

While the constructed $u_0^\delta$ preserves the mass constraint, the volume-filling constraint is no longer valid. To recover them we have to modify at least one of the other species. To this end, we make the following observation: in the set $ \Cdel$, it holds that
    \[
    \sum_{i=1}^n u_i= 1- u_0 \ge 1-\delta.
    \]
Therefore, denoting by $\displaystyle \overline{u}:=\mathop{\max}_{i=1,\ldots,n} u_i$, we necessarily have
    \begin{align} \label{eq:uk_bnd} \overline{u}\ge \frac{1- \delta}{n} \;\text{ in }\mathcal C_\delta \quad \mbox{ and } \quad     \overline{u} \le \sum_{ i=1}^n u_i = 1-u_0 \le 1- \frac{m_0}{|\Om|} \; \text{in } \Cavg.
    \end{align}
Let us now define for almost all $x\in \Omega$
$$
\bar k(x):= \min\big\{ k=1,\dots,n, \quad u_k(x) = \mathop{\max}_{i=1,\ldots,n} u_i(x) \big\} 
$$
so that $\overline{u}(x) = u_{\bar k(x)}(x)$ almost everywhere on $\Omega$. 
Let us then denote for all $k=1,\ldots,n$, $\mathcal C_\delta^k:= \{x\in \mathcal C_\delta, \; \bar k (x) = k \}$ so that $\displaystyle \mathcal C_\delta  = \bigcup_{k=1}^n \mathcal C_\delta^k$ and $\mathcal C_\delta^k \cap \mathcal C_\delta^{k'} = \emptyset$ as soon as $k\neq k'$. By definition, it then holds that 
\begin{equation}\label{eq:ubar}
u_k = \overline{u} \; \mbox{ in } \mathcal C_\delta^k.
\end{equation}
On the one hand, let us introduce
$$
m_{0,k}^\delta:= \int_{\mathcal C_\delta^k} (u_0^\delta - u_0)\,dx.
$$
Then, it holds that $0\leq m_{0,k}^\delta \leq \delta |\mathcal C_\delta^k|\leq \delta |\Omega|$. On the other hand, from the definition of $u_0^\delta$ and from the calculations above, it holds that 
$$
\int_{\mathcal C_{\rm avg}} (u_0 - u_0^\delta)\,dx = \int_{\mathcal C_\delta} (u_0^\delta - u_0)\,dx = \sum_{k=1}^n m_{0,k}^\delta.$$

\medskip

 Therefore, using Lemma~\ref{lem:aux}, there always exist measurable subsets $\mathcal C_{\rm avg}^{\delta,k}$ for $k=1,\ldots,n$ such that 
\begin{itemize}
    \item $\displaystyle \bigcup_{k=1}^n \mathcal C_{\rm avg}^{\delta,k} = \mathcal C_{\rm avg}$; 
    \item $\mathcal C_{\rm avg}^{\delta,k} \cap C_{\rm avg}^{\delta,k'} = \emptyset $ as soon as $k\neq k'$;
    \item $\int_{\mathcal C_{\rm avg}^{\delta,k}} (u_0 - u_0^\delta)\,dx = m_{0,k}^\delta = - \int_{\Cdel^k}(u_0 - u_0^\delta)\,dx$.
\end{itemize}

We can then define for all $k=1,\dots,n$ the perturbed function ${u_{k}^\delta}$ in the following way: 
\begin{align}
\label{eq:u_bar_k_pert}
    u_{k}^\delta=
    \begin{cases}
        u_{k}+ (u_0-u_0^\delta) & \text{in } \Cdel^k \cup \Cavg^{\delta,k} \\
        u_{k} & \text{otherwise.}
    \end{cases}
\end{align}
It can then be easily checked that $\|u_k^\delta - u_k \|_{L^\infty(\Omega)}\leq \epsilon$ for $\delta$ arbitrarily small. Moreover, as a consequence of \eqref{eq:uk_bnd} and \eqref{eq:ubar}, for $\delta$ small enough, it holds for all $k=1,\dots,n$,
\[
0 \leq \frac{1-\delta}{n} - \delta \leq u_{k}^\delta= u_{ k} + (u_0-u_0^\delta) \leq u_{k} \leq 1 ~ \text{in } \Cdel^k.
\]
On the other hand, using again \eqref{eq:uk_bnd} and \eqref{eq:ubar}, we can estimate
    \[ 
    0 \leq u_{k} \leq u_{k}^\delta = u_{k} + (u_0 -u_0^\delta) = u_{ k} + \lambda_\delta^*\Big(u_0-\frac{m_0}{|\Omega|}\Big) \le 1- \frac{m_0}{|\Omega|} + \lambda_\delta^* \Big(1 - \frac{m_0}{|\Omega|}\Big) ~ \text{in } \Cavg^{\delta,k}.  
    \]
Since $\displaystyle \lambda_\delta^* \mathop{\longrightarrow}_{\delta \to 0} 0$, choosing $\delta$ small enough implies that $\lambda_\delta^*(1-m_0/|\Omega|) < m_0/|\Omega|$ and therefore $u^\delta_{k} \le 1$ in $\Cavg^{\delta,k}$. Furthermore,
    \begin{align} \label{eq:uk_perb_mass}
    \int_\Omega u_{ k}^\delta \; dx = \int_\Omega u_{k} \; dx + \int_{\Cavg^{\delta,k} \cup \Cdel^k} (u_0-u_0^\delta)\;dx = \int_\Omega u_{ k} \;dx.
\end{align}
Finally, we observe that, almost everywhere in $\Omega$,
    \[
    \sum_{i=1}^nu_i^\delta = \sum_{i=1}^n u_i + (u_0 - u_0^\delta),
    \]
holds so that $\sum_{i=0}^n u_i^\delta = \sum_{i=0}^n u_i = 1$.  From this, we conclude that $\bbu^{\delta}:= (u_0^{\delta}, \dots, u_n^{\delta})$ lies in the set $ \A_m$ and satisfies $\|\bbu - \bbu^\delta\|_{L^\infty(\Omega)^{n+1}} \leq \epsilon$ for $\delta$ small enough.

\medskip

We now show that for $\delta \ll 1$, it holds $E(\bbu^{\delta})< E(\bbu)$ strictly, which gives us the desired contradiction. Firstly, let us note that $|\nabla u_0^\delta| \le |\nabla u_0|$, which gives
	\begin{align*}
E(\bbu^\delta) - E(\bbu)&\le  \into \sum_{i=0}^n \l((u_i^{\delta} \ln u_i^{\delta}- u_i \ln u_i)- (u_i^{\delta}- u_i) \r) +  \beta \l(u_0^{\delta}(1- u_0^{\delta})- u_0(1- u_0) \r) dx\\
\end{align*}
To estimate this difference, we first observe that \eqref{eq:u0_perb_mass}-\eqref{eq:uk_perb_mass} imply
	\begin{align*}
	\int_\Omega (u_i^\delta - u_i)\;dx = 0 \quad \text{for } i= 0,\ldots, n.
	\end{align*}
Using the convexity of $x \mapsto x \ln x $ and the concavity of $x \mapsto x(1-x)$ allows to further estimate
	\begin{equation}
	\label{est1}
	\begin{aligned}
	 E(\bbu^{\delta})- E(\bbu) & \le \sum_{k=1}^n  \int_{\Omega}  \l(\ln u_{k}^{\delta}+ 1\r) (u_{k}^\delta-u_{k})  \; dx+ \int_{\Omega} \l(\ln u_0^{\delta}+ 1\r) (u_0^{\delta}- u_0)\; dx  \\
	& \quad  + \beta \int_{\Omega} (1-2u_0)(u_0^{\delta}-u_0) \; dx \\
 & =: \Delta E_1 + \Delta E_2 + \Delta E_3. \\
	\end{aligned}
	\end{equation}
To estimate $\Delta E_1$ we first observe that 
    \begin{equation}
        \label{est-1.1}
    \begin{aligned}
      \Delta E_1 &=\sum_{k=1}^n \int_{\Omega}  \l(\ln u_{k}^{\delta}+ 1\r) (u_{k}^\delta-u_{ k})  \; dx=\sum_{k=1}^n \int_{\Cdel^k \cup \Cavg^{\delta,k}}  \l(\ln u_{k}^{\delta}+ 1\r) (u_0-u_{0}^\delta)  \; dx \\
      &= \sum_{k=1}^n \int_{\Cdel^k \cup \Cavg^{\delta,k}}  -\ln u_{k}^{\delta} (u_{0}^\delta-u_0)  \; dx.  
      \end{aligned}
    \end{equation}
The quantity $-\ln u_{ k}^{\delta}$ is nonnegative in $\Cdel^k \cup \Cavg^{\delta,k}$ while $(u_{0}^\delta-u_0)$ is nonnegative in $C_\delta^k$ and nonpositive in $\Cavg^{\delta,k}$, and furthermore \eqref{eq:uk_bnd} gives
\begin{align*}
    -\ln u_{k} = -\ln(\overline{u})\leq -\ln\l(\frac{1-\delta}{n}\r)\leq -\ln\l(\frac{1}{n}\l(1-\frac{m_0}{2|\Om|}\r)\r)  & \text{ in }\Cdel^k.
\end{align*}
Therefore \eqref{est-1.1} reduces to
\begin{align*}
	\Delta E_1 
 &\le -\ln\l(\frac{1}{n}\l(1-\frac{m_0}{2|\Om|}\r)\r) \int_{\Cdel}  (u_{0}^\delta-u_{0})  \; dx. 
	\end{align*}
In order to estimate the second term on the right-hand side of \eqref{est1} we first observe that in the set $\mathcal{C}_{\rm avg}$ it holds $\ln u_0^\delta  \ge \ln(m_0/|\Omega|)$, while due to the mass conservation we have
    \begin{equation}
        \label{sets}
     \int_{\Cdel} (u_0^\delta - u_0)\;dx = - \int_{\Cavg} (u_0^\delta - u_0)\;dx.
    \end{equation}
It follows that 
	\[
    \begin{split}
    \Delta E_2 &= \int_{\Omega} \ln u_0^{\delta} (u_0^{\delta}- u_0) \; dx \le \ln \delta \int_{\Cdel} (u_0^{\delta}- u_0) \; dx + \ln(m_0/|\Omega)| \int_{\Cavg} (u_0^\delta - u_0)\;dx \\
    &= (\ln \delta -\ln(m_0/|\Omega|)) \int_{\Cdel} (u_0^{\delta}- u_0) \; dx.
 \end{split}
	\]
 Finally, for the last term in \eqref{est1}  we use the fact that $-1 \le 1-2u_0 \le 1$ and \eqref{sets} again to have
 \begin{equation}
 \label{eq:est_delta_E3}
 \begin{aligned}
    \Delta E_3 & \le  \beta\int_\Omega   (1-2u_0)(u_0^\delta - u_0) \; dx \le \beta \int_\Omega |u_0^\delta - u_0| \;dx \\
    & \le \beta \int_{\Cdel}(u_0^\delta - u_0)\;dx - \beta \int_{\Cavg }(u_0^\delta - u_0)\;dx= 2\beta \int_{\Cdel}(u_0^\delta-u_0)\;dx.
\end{aligned}
 \end{equation}
 Summarising we have
 \begin{align*}
	E(\bbu^{\delta})- E(\bbu) & \le \l[ \ln \delta - \ln(m_0/|\Omega|) -  \ln\l(\frac{1}{n}\l(1-\frac{m_0}{2|\Om|}\r)\r) + 2\beta \r] \int_{\Cdel} (u_0^{\delta}- u_0) \; dx.
	\end{align*}
Thus, since $\displaystyle\int_{\Cdel} (u_0^{\delta}- u_0) \; dx > 0$ for all $\delta > 0$, taking $\delta $ sufficiently small so that the constant in front of this integral becomes negative yields the desired contradiction.
\medskip\\
\underline{Step 2: $u_0 \le 1- \delta_0$}. We argue again by contradiction and  assume that for all $\displaystyle \min\left( \epsilon, \frac{m_0}{2|\Om|} \right)\geq \delta>0$  the set
    \[
    \Cdelb= \{x \in \Om : \, u_0(x) > 1- \delta\}
    \]
is such that $|\Cdelb|>0$. We further define the set
    \[
        \Cavgb := \left\{ x \in \Omega \; : \; u_0(x) < \frac{m_0}{|\Omega|} \right\},
    \]
i.e. the part of $\Omega$ on which $u_0$ is strictly below its average. As before we argue that $|\Cdelb|>0$ for all $\delta >0$ arbitrarily small implies   $|\Cavgb| >0$. We then define, for $0 \leq \lambda \leq 1$, the perturbed function
    \[
    \tilde u_0^{\delta,\lambda}=
    \begin{cases}
        1-\delta & \text{in } \Cdelb \\
        \lambda u_0 + (1-\lambda) \frac{m_0}{|\Om|} & \text{in } \Cavgb \\
        u_0 & \text{otherwise},
    \end{cases}
\]
and arguing as in the previous step we show that for all $ \frac{\epsilon m_0}{4|\Omega|} >\delta>0 $ there exists  $\lambda_\delta^* \in (0,\epsilon/2)$ such that the function $\tilde u_0^\delta := \tilde u_0^{\delta,\lambda_\delta^*}$ has the same mass as $u_0$ and satisfies $\|u_0 - \tilde{u}_0^\delta\|_{L^\infty(\Omega)}\leq \epsilon$. 
We then observe that it holds
    \[
    \sum_{i=1}^n u_i= 1- u_0 \ge 1-\frac{m_0}{|\Om|} \quad\text{in } \Cavgb, 
    \]
which implies that the function $\displaystyle \overline{u}:= \mathop{\max}_{i=1,\ldots,n}u_i$ must satisfy 
    \begin{align*}
   \overline{u} \ge \frac{1}{n} \Big(1-\frac{m_0}{|\Om|}\Big) \;\text{in }\Cavgb \quad \mbox{ and } \quad   \overline{u} \le \sum_{ i=1}^n u_i = 1-u_0 \le \delta \;\text{in } \Cdelb.
    \end{align*}
For almost all $x\in \Omega$, we denote by $\displaystyle \tilde k(x):= \min\big\{ k=1,\ldots, n, \; u_k(x) = \mathop{\max}_{i=1,\ldots,n}u_i(x)\big\}$. Moreover, for all $k=1,\dots,n$ we denote by $\Cdelb^k:=\left\{ x \in \Cdelb, \; \tilde{k}(x) = k\right\}$ so that $\displaystyle \Cdelb = \bigcup_{k=1}^n \Cdelb^k$ and $\Cdelb^k \cap \Cdelb^{k'} = \emptyset $ as soon as $k\neq k'$. By definition, it then holds that 
$$
u_k = \overline{u} \; \mbox{ in }\Cdelb^k.
$$
By the mass conservation property
\begin{equation}\label{eq:mass}
\int_{\Cavgb}(u_0 - \tilde{u}_0^\delta)\,dx = \int_{\Cdelb} (\tilde{u}_0^\delta - u_0)\,dx,
\end{equation}
for all $\delta>0$ sufficiently small, using again Lemma~\ref{lem:aux}, there exist measurable subsets $\Cavgb^{\delta,k}\subset \Cavgb$ for all $k=1,\dots,n$ such that 
\begin{itemize}
    \item $\displaystyle \Cavgb = \bigcup_{k=1}^n \Cavgb^{\delta,k}$; \item $\Cavgb^{\delta,k} \cap \Cavgb^{\delta,k'} = \emptyset$ as soon as $k\neq k'$; 
    \item $\int_{\Cavgb^{\delta,k}}(u_0 - \tilde{u}_0^\delta) \,dx= \int_{\Cdelb^k}(\tilde{u}_0^\delta - u_0)\,dx$.
\end{itemize}

Thus, for $\delta$ sufficiently small, for all $k=1,\dots,n$, we define the function
    \[
    \tilde u_{k}^\delta=
    \begin{cases}
        u_{k}+ (u_0-u_0^\delta) & \text{in } \Cdelb^k \cup \Cavgb^{\delta,k} \\
        u_{k} & \text{otherwise}
    \end{cases}
    \]
which is such that $0\le \tilde{u}^\delta_{ k }\le 1$ and $\displaystyle \into \tilde u_{k}^\delta \; dx= \into u_{k} \; dx$. Using similar arguments as in Step~1,
 we obtain that for $\delta$ sufficiently small, $\tilde{\bbu}^{\delta}:= (\tilde u_0^{\delta}, \dots, \tilde u_n^{\delta}) \in \A_m$ and $\|\tilde{\bbu}^\delta - \bbu\|_{L^\infty(\Omega)^{n+1}} \leq \epsilon$. Moreover, it holds that
 	\begin{equation*}
	\begin{aligned}
	 E(\tilde{\bbu}^{\delta})- E(\bbu) & \le  \sum_{k=1}^n \int_{\Omega}  \ln \tilde{u}_{k}^{\delta} (\tilde{u}_{k}^\delta-u_{k})  \; dx+ \int_{\Omega} \ln \tilde{u}_0^{\delta} (\tilde{u}_0^{\delta}- u_0)\; dx  \\
	& \qquad  + \beta \int_{\Omega} (1-2u_0)(\tilde{u}_0^{\delta}-u_0) \; dx =: \Delta E_1 + \Delta E_2 + \Delta E_3. 
	\end{aligned}
	\end{equation*}
 To estimate $\Delta E_1$ we first observe that
\begin{align*}
   -\ln u_{k}  \begin{cases} 
   \le -\ln\l(\frac{1}{n}(1-\frac{m_0}{|\Om|})\r)  & \text{ in }\Cavgb^{\delta, k}\\
   \ge -\ln \delta & \text{ in }\Cdelb^k.
  \end{cases}
\end{align*}
Thus we can estimate
\begin{align*}
	\Delta E_1 &=\sum_{k=1}^n \int_{\Omega}  \ln \tilde{u}_{ k}^{\delta} (\tilde{u}_{k}^\delta-u_{\bar k})  \; dx =\sum_{k=1}^n \int_{\Cavgb^{\delta,k} \cup \Cdelb^k}  \ln \tilde{u}_{ k}^{\delta} (\tilde{u}_{k}^\delta-u_{\bar k})  \; dx\\
 &=-\sum_{k=1}^n \left[\int_{\Cdelb^k}  \ln \tilde{u}_{k}^{\delta} (\tilde{u}_{0}^\delta-u_{0})  \; dx - \int_{\Cavg^{\delta,k}}  \l(\ln \tilde{u}_{k}^{\delta}+ 1\r) (\tilde{u}_{0}^\delta-u_{0})  \; dx\right]\\
 &\le \l[ -\ln \delta + \ln\l(\frac{1}{n}(1-\frac{m_0}{|\Om|})\r)\r]\int_{\Cdel}  (\tilde{u}_{0}^\delta-u_{0})  \; dx ,
	\end{align*}
where we used \eqref{eq:mass} to obtain the last inequality.

The second integral $\Delta E_2$ in \eqref{est1} can be estimated via
	\[
	\Delta E_2 = \int_{\Omega} \ln \tilde{u}_0^{\delta} (\tilde{u}_0^{\delta}- u_0) \; dx \le \ln (1-\delta) \int_{\Cdelb} (\tilde{u}_0^{\delta}- u_0) \; dx + \ln(m_0/|\Omega)|\int_{\Cavgb} (\tilde{u}_0^\delta - u_0)\;dx,
	\]
where we used $\ln(\tilde{u}_0^\delta) \le \ln(m_0/|\Omega|)$ in $\Cavgb$. Using again \eqref{eq:mass},
we obtain
 	\[
	\Delta E_2 \le (\ln (1-\delta) -\ln(m_0/|\Omega|)) \int_{\Cdel} (\tilde{u}_0^{\delta}- u_0) \; dx.
	\]
Estimating the concave part $\Delta E_3$ as in \eqref{eq:est_delta_E3} and collecting all terms eventually yields
\begin{align*}
E(\tilde \bbu^{\delta})- E(\bbu) & \le \l[ - \ln \delta +  \ln\l(\frac{1}{n}\l(1-\frac{m_0}{|\Om|}\r)\r) +\ln (1-\delta) -\ln(m_0/|\Omega|)+ 2\beta \r] \int_{\Cdelb} (\tilde{u}_0^{\delta}- u_0) \; dx.
	\end{align*}
As the integral on the right-hand side is strictly negative but the coefficient in front of it becomes positive for $\delta$ sufficiently small, we again reach the desired contradiction. 
\medskip
 
\underline{Step 3: $\delta \le u_i$, $i=1,\ldots n$}. 
To fix the ideas let us assume that $i=1$. While the proof uses the same construction as in Step 1, it is crucial to make sure that the index $\bar k(x)$ used to construct the sets $\Cdel^k$ and the functions such as in \eqref{eq:u_bar_k_pert} is such that $\bar k(x)$ is never equal to $0$, as applying  \eqref{eq:u_bar_k_pert} to define $u_0^\delta$ with $k=0$ would yield a function $u_0^\delta$ which does not belong to $H^1(\Omega)$ and thus renders the value of the energy to be infinity. To this end, we may use the upper bound on $u_0$ established at Step~2 of the proof to calculate
\begin{align*}
    \sum_{j\neq 0,1} u_j = 1-u_0 - u_1 \ge (\delta_0 - \delta) \quad \text{in the set } \mathcal{C}_{\delta,1}:= \{ x\in \Omega \; : \; u_1(x) < \delta \}.
\end{align*}
As $\delta_0$ from Steps 1 and 2 is fixed at this point, choosing $\delta < \delta_0$ we can go on from here to ensure that, defining $\bar k(x):=\min\{ k=2,\ldots,n, \quad u_k(x) = \max_{i=2,\ldots,n} u_i(x)\}$, we have that $u_{\bar k(x)}(x)> (\delta_0 - \delta)/n$ almost everywhere in $\mathcal{C}_{\delta,1}$. Then, arguing as in Step 1 gives the existence of $\delta_1 > 0$ such that $\delta_1 \le u_1$. As the choice $i=1$ was arbitrary it follows that the argument can then be applied to any other $u_j$, using the same construction as just done for $u_1$. 

\medskip

\underline{Step 4: Conclusion}.
We then observe that the parameter $\delta := \min(\delta_0, \delta_1, \dots, \delta_n)$ satisfies all the properties in the statement, and therefore the conclusion follows.
\end{proof}

Thanks to these uniform bounds and arguing by elliptic regularity, we can derive first order optimality conditions as given by the following theorem.
\begin{theorem}
    \label{thm:optimality}
    Let $\bbu \in \cA_m$ be a local minimizer of the energy \eqref{eq:energy} in the $L^\infty(\Omega)^{n+1}$ topology. Then $u_0 \in H^2(\Om)$ 
    and is solution to
    \begin{equation}
    \label{eq:u0}
    \begin{aligned}
        - \eps \Delta u_0 &= \ln \frac{1-u_0}{u_0} - \beta (1-2u_0) - \frac{1}{|\Om|}\into \biggl( \ln \frac{1-u_0}{u_0} - \beta (1-2u_0) \biggr) \, dx, \qquad && \mbox{ in }\Omega,\\     
        \frac{\partial u_0}{\partial n} &= 0, && \mbox{ on }\partial \Omega.
    \end{aligned}
    \end{equation}
Moreover, it holds
    \begin{equation}
        \label{eq:ui}
        u_i = \frac{m_i}{|\Om|-m_0} (1-u_0), \quad i=1,\dots,n.
    \end{equation}
In addition, $\bbu $ is an element of $ (C^{\infty}(\overline{\Om}))^{n+1}$ and a classical solution to \eqref{eq:stationary}.     
\end{theorem}

\begin{proof}
\underline{ Step~1: establishing \eqref{eq:u0} and \eqref{eq:ui}}. Given $\bpsi := (\psi_1, \ldots, \psi_n) \in  (C^{\infty}(\overline{\Om}))^n$, set
	\begin{equation}
	    \label{ni}
	\nu_i:= \psi_i- \frac{1}{|\Om|} \into \psi_i \;dx \quad \text{for all } i= 1, \dots, n.
	\end{equation}
Fix now $r>0$, consider the perturbations 
	\[
	u_{r, i}:= u_i+ r \nu_i \quad \text{for all } i= 1, \dots, n,
	\]
and set
	\begin{equation*}
	u_{r,0}:= 1- \sum_{i=1}^n u_{r, i}= 1- \sum_{i=1}^n u_i- r \sum_{i=1}^n \nu_i = u_0+ r\nu_0, \quad \text{with $\displaystyle \nu_0:= -\sum_{i=1}^n \nu_i$}.
 	\end{equation*}
We then set $\bbnu:= (\nu_0, \dots, \nu_n)$, from which
$\bbu_r= \bbu+ r\bbnu$ and in turn $\bbu_r \in \A_m$  for $r$ sufficiently small, due to the strict lower and upper bounds of minimizers shown in Theorem~\ref{adm-min}, so that $E(\bbu) \le E(\bbu_r)$. We now want to calculate the variation of $E$, that is,
   \[
   \lim_{r \to 0} \frac{E(\bbu_r)- E(\bbu)}{r}.
   \]
Note that, thanks again to the strict bounds of Theorem~\ref{adm-min}, we obtain uniform $L^{\infty}$-bounds on $\ln u_i, \ln u_{r,i}$, for any $i=0,\dots,n$ and for $r$ sufficiently small. Therefore, we can use dominated convergence to calculate
	\[
 0 \le \lim_{r \to 0} \frac{E(\bbu_r)- E(\bbu)}{r}= 
	 \into \sum_{i=1}^n \biggl[\biggl(\ln u_i- \ln u_0 - \beta (1- 2 u_0) \biggr) \nu_i   - \eps \nabla u_0\cdot \nabla \nu_i\biggr] \;dx,
	\]
for all such $\nu_i$. We repeat the same argument for $-\bbnu$ and then use \eqref{ni} to eventually infer, using Fubini theorem, 
	\begin{align*}
	0&= 
 \into \sum_{i=1}^n \biggl[ \biggl(\ln u_i- \ln u_0- \beta (1- 2 u_0)\biggr)\psi_i - \eps \nabla u_0\cdot \nabla \psi_i \biggr] dx \\
	& \qquad - \into \biggl(\frac{1}{|\Om|}\into \ln u_i- \ln u_0 - \beta (1- 2 u_0) \; dy \biggr) \psi_i \; dx
	\end{align*}
for every $\bpsi=(\psi_1, \ldots, \psi_n)\in (\mathcal C^\infty(\overline{\Omega}))^n$. Setting 
\begin{equation*}
    \lambda_i = \frac{1}{|\Om|}\into \ln u_i- \ln u_0 - \beta (1- 2 u_0) \; dy,
\end{equation*}
we obtain that for any $i=1,\dots,n$, for any $\varphi \in H^1(\Om)$, 
\begin{equation}
	\label{eq:EulLag-weak}
 	\int_\Omega \Big[\ln u_i- \ln u_0 - \beta (1- 2 u_0)\Big]\varphi\;dx - \eps \int_\Omega \nabla u_0\cdot \nabla \varphi\;dx= \lambda_i\int_\Omega \varphi\;dx.
	\end{equation}
Thanks to the uniform bounds of Theorem~\ref{adm-min}, we know that
\begin{align*}
    \ln u_i - \lambda_i - \ln u_0 - \beta (1- 2u_0) \in L^2(\Omega) \quad \text{for all } i=1, \dots, n,
\end{align*} 
therefore, using standard elliptic regularity theory (see, e.g., \cite{troianiello2013elliptic}), we get from \eqref{eq:EulLag-weak} that $u_0 \in H^2(\Omega)$. This in turn implies that the Euler-Lagrange system is satisfied in strong form as follows
	\begin{equation}
	    \label{eq:EulLag-strong}
	-\eps \Delta u_0= \ln u_i -\lambda_i - \ln u_0 - \beta (1- 2u_0) \quad \text{a.e. in } \Om, \text{for all }  i= 1, \dots, n.
	\end{equation}
Let us now proceed to eliminate the Lagrange multipliers from the equations. We write \eqref{eq:EulLag-strong} for any pair of indexes $i\ne k \in \{1, \dots, n\} $ and take the difference of the corresponding expressions. This gives
	\begin{equation}
    \label{eq:wi-wj}
	\ln u_i- \ln u_k= \lambda_i- \lambda_k,  	\end{equation}
that is, 
    \begin{equation*}
        u_k= u_i e^{\lambda_k- \lambda_i}.
    \end{equation*}
Solving this equation for $u_i$ by  using the volume-filling constraint gives
	\begin{equation}
    \label{eq:ui-interm}
	  u_i= \frac{e^{\lambda_i}}{\sum\limits_{k=1}^n e^{\lambda_k}} (1-u_0).
	\end{equation}
Inserting \eqref{eq:ui-interm} in \eqref{eq:EulLag-strong}, the Euler-Lagrange system reduces to the following PDE on $u_0$ together with $n$ Lagrange multipliers
\begin{equation}
    \label{eq:eul-lag-u0}
    - \eps \Delta u_0 = \ln \frac{1-u_0}{u_0} - \beta (1-2u_0) - \ln \l( \sum_{k=1}^n e^{\lambda_k} \r), \quad \mbox{ in }\Omega,
\end{equation}
together with the boundary condition 
\begin{equation}\label{eq:BC}
\frac{\partial u_0}{\partial n} = 0, \quad \mbox{ on }\partial \Omega.
\end{equation}
Integrating \eqref{eq:eul-lag-u0} over $\Omega$, together with \eqref{eq:BC}, gives 
\begin{equation*}
    \ln \l( \sum_{k=1}^n e^{\lambda_k} \r) = \frac{1}{|\Om|} \into \ln \frac{1-u_0}{u_0} - \beta (1-2u_0) \, dx =: \lambda_0.
\end{equation*}
Then integrating \eqref{eq:ui-interm} over $\Om$ we obtain that 
\[ e^{\lambda_i} = \frac{m_i}{|\Om|-m_0} e^{\lambda_0}, \quad i=1,\dots,n,\]
and in turn \eqref{eq:ui}.

\underline{Step~2: establishing \eqref{eq:stationary}}.
First, it follows from \eqref{eq:u0} combined with the uniforms bounds on $u_0$ that, by elliptic regularity, $u_0$ is smooth in $\Om$, and as a consequence of \eqref{eq:ui}, $\bbu$ is smooth as well. Secondly, it is enough to show that $u_i$ satisfies \eqref{eqi} for $i=1,\dots,n$, since then \eqref{eq0} is automatically satisfied thanks to the volume-filling constraint. Remark that, rewriting \eqref{eq:EulLag-strong} and \eqref{eq:wi-wj} with the notations \eqref{eq:wi}-\eqref{eq:w0} leads to
\begin{equation*}
\begin{aligned}
        w_i - w_0 &= \lambda_i, \quad && i=1,\dots,n, \\
        w_i - w_j &= \lambda_i - \lambda_j, && i \ne j=1,\dots,n.
\end{aligned}
\end{equation*}
Taking into account that $\bbu$ is independent of time and using the previous relations, we obtain from \eqref{eqi} that
\begin{align*}
    \divergenz \l( \sum_{0 \leq j \neq i \leq n} K_{ij} u_i u_j \nabla (w_i-w_j) \r)& = \divergenz \l( \sum_{1 \leq j \neq i \leq n} K_{ij} u_i u_j \nabla (w_i-w_j) + K_{i0} u_i u_0 \nabla (w_i-w_0) \r) \\
    &= \divergenz \l( \sum_{1 \leq j \neq i \leq n} K_{ij} u_i u_j \nabla (\lambda_i-\lambda_j) + K_{i0} u_i u_0 \nabla \lambda_i \r) \\
    &= 0,
\end{align*}
which concludes the proof.
\end{proof}

Let us make a few  comments on the result that we have established here. First, we have proved that any local minimizer of the energy functional $E$ for the $L^\infty(\Omega)^{n+1}$ topology is indeed a solution to the stationary problem \eqref{eq:stationary}. Of course, we do not expect the converse to be true, in the sense that there may exist other solutions to \eqref{eq:stationary} which are not local minimizers of the energy functional $E$. However, because of the gradient flow structure of the time-dependent system \eqref{eq:system_intro}, it is natural to conjecture that solutions to this system will converge in the long-time limit to some stationary states that are local minimizers of $E$. We are not able to prove this claim here, but give some numerical evidence in Section~\ref{sec:simulations}. 

Second, thanks to Theorem~\ref{thm:optimality}, we can study the properties of the local minimizers of $E$ by studying the scalar equation \eqref{eq:u0}, which is now the first-order Euler-Lagrange equation of the \emph{single-species} Cahn-Hilliard energy
\begin{equation*}
    E_0(u_0) = \into \left[ u_0 \ln u_0 + (1-u_0) \ln (1-u_0) + \frac{\eps}{2} |\nabla u_0|^2 + \beta u_0(1-u_0)\right] \;dx .
\end{equation*}


\medskip

\section{Convexity properties and long-time behaviour}
\label{sec:conv-long}

In this section we present some results obtained on the behaviour of solutions to the time-dependent system \eqref{eq:system_intro} in the case when the parameters $\varepsilon$ and $\beta$ are chosen such that the functional $E$ is convex. While this range of parameter values may not be practically relevant for some physical applications where separation effects dominate over diffusion, we nevertheless believe that the present analysis is instructive and might be seen as a useful preliminary step towards the study of the long-time behaviour of solutions to \eqref{eq:system_intro} in the general case. We first give explicit conditions on the parameters $\varepsilon$ and $\beta$ for $E$ to be convex. We then prove that, in a stable regime, solutions of the time-dependent system converge exponentially fast to the minimizer of $E$, which is proved to be unique in this setting.

\subsection{Convexity properties}

Let us now study convexity properties of the energy \eqref{eq:energy}. We begin by recalling the famous Poincaré-Wirtinger inequality: there exists a constant $C_p>0$ such that for all $v\in H^1(\Omega)$, 
\begin{equation}
\label{eq:poincare}
    \int_\Omega \Big(v - \frac{1}{|\Omega|}\int_\Omega v\; dy\Big)^2\;dx \le C_p \int_\Omega |\nabla v|^2\;dx,
\end{equation}
as well as the logarithmic-Sobolev inequality on bounded domains, \cite[Lemma 1]{Desvillettes2014},
\begin{equation}\label{eq:log_sob}
\int f\log f \;dx \le C_{\mathrm{sob}}\int_\Omega |\nabla \sqrt{f}|^2\;dx,\text{ for every $f \in H^1(\Omega; \R_+)$ s.t. } \int_\Omega f\;dx = 1,
\end{equation}
where the constant $C_{\mathrm{sob}}$ depends on $d$ and $\Omega$, only.
We then have the following lemma:
\begin{lemma}[Convexity of the free energy]
Let $n \geq 1$ and $C_p$ be the Poincaré constant in \eqref{eq:poincare}. Assume that
\begin{equation}
    \label{eq:convexity-cond}
    \frac{1}{2m_0} + \frac{1}{|\Omega|} \Big(\frac{\eps}{2 C_p} - \beta\Big) \geq 0.
\end{equation}
Then, the energy functional \eqref{eq:energy} is convex in the set $\cA_m$. If $n=1$, the result holds under the weaker condition 
\begin{equation}
    \label{eq:convexity-condition-onespecies}
 \frac{|\Om|}{2m_0(|\Om|-m_0)} + \frac{1}{|\Om|} \Big(\frac{\eps}{2C_p}- \beta \Big) \ge 0.  
\end{equation}
Moreover, whenever the inequalities are strict, the energy is strictly convex.
\end{lemma}
\begin{proof}
Let $\bbu:=(u_0, \ldots, u_n), \bbv:=(v_0, \ldots, v_n) \in  \mathcal{A}_m$. This in particular implies that
    \[
    \into u_i \; dx= \into v_i \; dx \quad \text{for all } i= 0, \dots, n.
    \]
We want to estimate from below the quantity
\[E(\bbu)- E(\bbv)-  E'(\bbv)(\bbu- \bbv) 
        = \sum_{i=0}^n \into  u_i \ln \frac{u_i}{v_i}+ \frac{\eps}{2} |\nabla (u_0- v_0)|^2- \beta(u_0- v_0)^2 \; dx.  \]
The first terms can be  estimated thanks to the Csiszár–Kullback–Pinsker inequality as
\[ \into  u_i \ln \frac{u_i}{v_i} \;dx \geq \frac{1}{2m_i} \l(\into |u_i-v_i| \;dx \r)^2, \quad i=0,\dots,n, \]
while the gradient term is  estimated using the Poincaré inequality \eqref{eq:poincare}.
We obtain
\begin{equation}
    \label{dis-en}    
\begin{split}
        & E(\bbu)- E(\bbv)- E'(\bbv)(\bbu- \bbv)  \\
        & \geq \sum_{i=0}^n \frac{1}{2 m_i} \left(\into |u_i-v_i|dx \right)^2 + \Big(\frac{\eps}{2C_p}-\beta\Big) \l(\into |u_0-v_0|^2 dx\r) \\
        &\geq  \sum_{i=0}^n \frac{1}{2 m_i} \left(\into |u_i-v_i|dx \right)^2 + \frac{1}{|\Om|}\Big(\frac{\eps}{2C_p}-\beta\Big) \l(\into |u_0-v_0| dx\r)^2,
 \end{split}
 \end{equation}
where we applied Jensen's inequality. If $n=1$, using $u_1 = 1-u_0, m_1 = |\Om|-m_0$, we obtain
\[
\begin{split}
    E(\bbu)- E(\bbv)-  E'(\bbv)(\bbu- \bbv)  \ge \l(\frac{|\Om|}{2m_0(|\Om|-m_0)} + \frac{1}{|\Om|} \Big(\frac{\eps}{2C_p}- \beta \Big) \r) \left(\int_\Omega |u_0-v_0|\;dx\right)^2,
\end{split}
\]
hence condition \eqref{eq:convexity-condition-onespecies} ensures convexity. However, when $n \geq 2$, one cannot easily take advantage of the terms $\displaystyle \sum_{i=1}^n \frac{1}{2m_i} \left(\into |u_i-v_i|dx \right)^2$. Therefore, using only the fact that they are nonnegative, one obtains from \eqref{dis-en} that
\[
\begin{split}
    E(\bbu)- E(\bbv)-  E'(\bbv)(\bbu- \bbv)  \ge \l(\frac{1}{2m_0} + \frac{1}{|\Om|} \Big(\frac{\eps}{2C_p}- \beta \Big) \r) \left(\int_\Omega |u_0-v_0|\;dx\right)^2,
\end{split}
\]
which gives again the convexity of $E$ under condition \eqref{eq:convexity-cond}.
%
\end{proof}

When applicable, it follows from strict convexity that both the minimization problem and the Euler-Lagrange equation \eqref{eq:u0} have a unique solution. Moreover, the constant state $u_0=\frac{m_0}{|\Om|}$ always solves \eqref{eq:u0}. This leads to the following corollary. 
\begin{corollary}[Uniqueness of the minimizer] 
\label{cor:uniqueness}
Whenever condition
\begin{align*}
\frac{1}{2m_0} + \frac{1}{|\Omega|} \Big(\frac{\eps}{2 C_p} - \beta\Big) > 0
\end{align*}
holds, the minimization problem \eqref{eq:min} has a unique solution, given by the constant states
\begin{equation}
        \label{eq:u_const}
        u_i^\infty= \frac{m_i}{|\Om|} \quad \text{for all } i= 0, \dots, n.
\end{equation}
\end{corollary}

 \subsection{Large-time asymptotics in the stable regime}
We discuss now the large-time asymptotics of the evolution system \eqref{eq:system_intro} in the framework of Corollary~\ref{cor:uniqueness}, i.e. when the energy admits a unique (constant) minimizer. We aim to show that $\bbu$ converges exponentially fast to this unique minimizer, for arbitrary large initial data with finite relative energy.

\medskip

It follows from \cite[Theorem 2.1]{EMP2021} that, for any initial condition in $\cA_m$, the system admits a weak solution $\bbu= (u_0, u_1, \dots, u_n)$ that satisfies the constraints \eqref{eq:constraint} and such that (see \cite[Definition 2.1]{EMP2021} for more details),
\begin{itemize}
    \item $u_i \in L^2((0,T); H^1(\Omega))$ for all $i=1,\dots,n$;
    \item $u_0 \in L^2((0,T); H^2(\Omega))$;
    \item  $ \partial_t u_i \in L^2((0, T); (H^{1}(\Omega))')$ for all $i=0,\dots,n$;
\end{itemize}

Let $\bbu^\infty:=(u_0^\infty,\ldots, u_n^{\infty})$ be defined by \eqref{eq:u_const}. The relative energy functional is defined via the Bregman divergence as, for all $\bbu:=(u_0, \ldots, u_n) \in \cA_m$,
	\begin{equation}
 	\label{rel-entr}
 \begin{aligned}
	RE[\bbu \, | \, \bbu^{\infty}] &= E(\bbu)- E(\bbu^{\infty})-  E'(\bbu^{\infty})(\bbu- \bbu^{\infty}) \\
 &= \into \left[\sum_{i=0}^n u_i \ln \frac{u_i}{u_i^\infty}+ \frac{\eps}{2} |\nabla(u_0- u_0^\infty)|^2- \beta (u_0- u_0^\infty)^2\right] \;dx.
 \end{aligned}
	\end{equation}
Our result is stated in the next theorem.
\begin{theorem}[Global exponential stability of the constant steady state]
\label{thm:convergence}
Let $\bbu^{\infty}$ be given by \eqref{eq:u_const} and let $\bbu^0 \in \cA_m$ be an initial condition such that $\bbu^0 \in H^2(\Omega)^{n+1}$ and
\begin{equation*}
	RE[\bbu^0 \, | \, \bbu^{\infty}]< \infty.
\end{equation*}
Let $\bbu$ be a weak solution to \eqref{eq:system_intro} as constructed in \cite{EMP2021}.
Then, $\bbu$ satisfies
\begin{equation}
    \label{eq:conv-exp}
    RE[\bbu(t) \, | \, \bbu^{\infty}]  \leq e^{-\lambda t} RE[\bbu^0 \, | \, \bbu^{\infty}], ~ t > 0,
\end{equation}
with the rate
\begin{equation}
    \label{eq:rate}
    \lambda = 4k \min \l(\frac{1}{C_{\rm sob}},\frac{1}{C_p}-\frac{2 \beta}{\eps} \r)
\end{equation}
with $k:= \min_{0 \le i \ne j \le n} K_{ij} >0$, $C_p$ and $C_{\rm sob}$ being the constants in  \eqref{eq:poincare} and \eqref{eq:log_sob}, respectively. Furthermore, under the condition
\begin{equation}
    \label{eq:cond-global-stab}
    \frac{\eps}{2 C_p} - \beta > 0,
\end{equation}
it holds
\[
    \|u_i(t)- u_i^\infty\|^2_{L^1(\Om)} \leq RE[\bbu(t) \, | \, \bbu^{\infty}] \leq e^{-\lambda t} RE[\bbu^0 \, | \, \bbu^{\infty}] \to  0 \quad \text{ as $t \to +\infty$, } \text{for all } i= 0, 1, \dots, n.
    \]
\end{theorem}

\begin{remark}
    The global stability condition \eqref{eq:cond-global-stab} is stronger than the convexity condition on the energy \eqref{eq:convexity-cond} as it does not take the mass into account.  
\end{remark}

We recall here some results proved in~\cite{EMP2021} and in particular the construction of a weak solution to the time-dependent system. The method consists in introducing a time-discrete regularised version of the evolution system, together with a suitable weak formulation. More precisely, the following theorem holds:

\begin{theorem}[Theorem 3.1 from \cite{EMP2021}]\label{thm:existence_time_discrete}
Define the sets
$$
\mathcal A:= \left\{\bu:= (u_i)_{1\leq i \leq n} \in (L^\infty(\Om))^n: \, u_i \ge 0, \, i=1,\ldots, n, \;   u_0:= 1 - \sum_{i=1}^n u_i  \geq 0\right\}
$$
and  
$$
\mathcal{B} := \left\{ \bphi := (\phi_i)_{1\leq i \leq n} \in (L^\infty(\Omega))^n \; : \; \phi_0:= - \sum_{i=1}^n \phi_i \in H^1(\Omega) \right\},
$$
 and let $\tau >0$ be a discrete time step, let $p\in \mathbb{N}$, and let $\bu^p \in \mathcal A \cap (H^2(\Omega))^n$. Then, there exists a solution   $(\bu^{p+1}, \bar{{\bw}}^{p+1}) \in (\mathcal A \cap (H^2(\Omega))^n) \times (H^2(\Omega))^n$ to the following coupled system: for all $1\leq i \leq n$, for all $\phi_i \in  H^2(\Omega)$,
 	\begin{equation}
	\label{dis-sys1}
	\begin{split}
	\into \frac{u_i^{p+1}- u_i^p}{\tau} \phi_i \;dx &= -\into \biggl(\sum_{1 \le j \ne i \le n} K_{ij} u_i^{p+1} u_j^{p+1} \nabla (\bar{w}_i^{p+1}- \bar{w}_j^{p+1})+ K_{i0} u_i^{p+1} u_0^{p+1} \nabla \bar{w}_i^{p+1} \biggr) \cdot \nabla \phi_i \;dx \\
	& \quad -\tau \langle \bar{w}_i^{p+1}, \phi_i \rangle_{H^2(\Omega)},
	\end{split}
	\end{equation}
 and for all $\bpsi = (\psi_i)_{1\leq i \leq n} \in \mathcal B \cap (L^\infty(\Omega))^n  $,  
\begin{equation}\label{eq:w_to_u_weak0}
\sum_{i=1}^n \into (\ln u^{p+1}_i- \ln u^{p+1}_0)\psi_i + \eps \nabla u^{p+1}_0\cdot \nabla \psi_0 dx = \sum_{i=1}^n \into \left( \bar w^{p+1}_i + \beta(1- 2 u_0^p)\right)\psi_i dx,
\end{equation}
where $u_0^p =  1 -\sum\limits_{i=1}^n u_i^p$.
 Moreover, the function $\bu^{p+1}$ satisfies the following property: there exists $\delta_p>0$ such that 
\begin{equation*}
u_i^{p+1}\geq \delta_p, \quad \mbox{ for all }1\leq i \leq n, \;\mbox{ and }  \;u_0^{p+1}:= 1 - \sum_{i=1}^n u_i^{p+1} \geq \delta_p,\quad\text{ a.e. in } (0,T)\times\Omega.
\end{equation*}
\end{theorem}
\begin{proof}[Proof of Theorem \ref{thm:convergence}]
Let $\bbu^0:=(u_0^0, \ldots, u_n^0)\in \cA_m\cap H^2(\Omega)^{n+1}$. For a given value of time step $\tau>0$, and starting from $\bu^0:=(u_1^0, \ldots, u_n^0)$, we consider the sequence $(\bu^p)_{p\in \mathbb{N}}$ of discrete iterates given by Theorem \ref{thm:existence_time_discrete}. For all $p\in \mathbb{N}$, we then define $\bbu^p = (u_0^p, \bu^p)$ where $u_0^p:= 1 - \sum_{i=1}^n u_i^p$. We first note that they enjoy enough regularity to define the following quantities for all $p\geq 0$, using $\overline{\bw}^{p+1}:=(\overline{w}_1^{p+1}, \ldots, \overline{w}_n^{p+1})$,
\begin{align*}
w_0^{p+1/2}& := -\eps \Delta u_0^{p+1} + \beta (1 - 2 u_0^p),\\
w_i^{p+1}& := \overline{w}_i^{p+1} + w_0^ {p+1/2}, \quad \mbox{ for }i=1,\ldots,n,\\
\bbw^{p+1} &:= (w_i^{p+1})_{1\leq i \leq n}.
\end{align*}
Then, one can define the piecewise constant interpolation $\bbu^{(\tau)}$ as well as its time-shifted version $\sigma_\tau \bbu^{(\tau)}$: for all $p\in \mathbb{N} \setminus \{0\}$, define the discrete time $t_p=p\tau$ and, for all $t \in (t_{p}, t_{p+1}]$, 
	\begin{equation*} 
	\bbu^{(\tau)}(t)  = \bbu^{p+1}, \, \sigma_\tau \bbu^{(\tau)}(t) = \bbu^{p-1}, \,  \bbw^{(\tau)}(t)   = \bbw^{p+1},
	\end{equation*}
together with $\bbu^{(\tau)}(0) =\bbu^0$. For any $t \in (0, \infty)$, we also define $P^{(\tau)}(t)\in \mathbb{N} \setminus \{0\}$ to be the lowest integer such that $t_{P^{(\tau)}}\geq t$. Choosing $\phi_i= \bar w_i^{p+1} = w_i^{p+1}- w_0^{p+1/2}$ in \eqref{dis-sys1} one obtains a discrete (relative) energy-energy dissipation inequality and the authors showed in the proofs of \cite[Lemmas 4.2, 4.3, 4.4]{EMP2021} that the terms in the dissipation can be estimated, which yields the following inequality, for all $t>0$,
\begin{equation*}
    \begin{split}
    & RE[\bbu^{(\tau)}(t) \, | \, \bbu^{\infty}]-RE[\bbu^0 \, | \, \bbu^{\infty}] \\
    & \leq - k \sum_{i=0}^n \int_0^t \into \frac{|\nabla u_i^{(\tau)}|^2}{u_i^{(\tau)}} \;dx ds+ \sum_{p=0}^{P^{(\tau)}(t)-1} 4k\beta \tau \|\nabla u_0^{p+1}\|_{L^2(\Om)} \|\nabla u_0^p\|_{L^2(\Om)} \\
   & \quad -  2k \eps \int_0^t \into |\Delta u_0^{(\tau)}|^2 \;dx ds
    - k \int_0^t \into u_0^{(\tau)}(1- u_0^{(\tau)}) |\nabla w_0^{(\tau)}|^2 \;dx ds -\tau \int_0^t \sum_{i=1}^n   \| w_i^{(\tau)}- w_0^{(\tau)} \|^2_{H^2(\Omega)} \; ds.
    \end{split}
\end{equation*}
Considering the nonnegativity of the two last terms, and using the inequality
\[\sum_{p=0}^{P^{(\tau)}(t)-1} 4k\beta \tau \|\nabla u_0^{p+1}\|_{L^2(\Om)} \|\nabla u_0^p\|_{L^2(\Om)} \leq 2k \beta \int_0^t \l(\|\nabla u_0^{(\tau)}\|^2_{L^2(\Om)} + \|\nabla\sigma_\tau u_0^{(\tau)}\|^2_{L^2(\Om)}\r) ds, \]
the energy inequality rewrites
\begin{equation*}
    \begin{split}
    RE[\bbu^{(\tau)}(t) \, | \, \bbu^{\infty}]- RE[\bbu^0 \, | \, \bbu^{\infty}] & \leq - 4k \sum_{i=0}^n \int_0^t \into |\nabla \sqrt{u_i^{(\tau)}}|^2 \;dx ds -  2k \eps \int_0^t \into |\Delta u_0^{(\tau)}|^2 \;dx ds\\
   & \quad + 2k \beta \int_0^t \l(\|\nabla u_0^{(\tau)}\|^2_{L^2(\Om)} + \|\sigma_\tau \nabla u_0^{(\tau)}\|^2_{L^2(\Om)}\r)ds.
    \end{split}
\end{equation*}
We use the logarithmic Sobolev inequality \eqref{eq:log_sob}, together with the following estimate, which we prove for any $v\in H^2(\Omega)$ with $\partial_n v = 0$ on $\partial \Omega$: denoting by $m= \frac{1}{|\Omega|}\int_\Omega v\,dx$, it holds that
 \begin{align*}
     \int_\Omega |\nabla v |^2 \;dx &= \int_\Omega |\nabla (v -m)|^2 \;dx = -\int_\Omega (u-m)\Delta (u-m) \le \left(\int_\Omega (u-m)^2\right)^{1/2}\left(\int_\Omega (\Delta(u-m))^2\right)^{1/2}\\
     &\le \left(C_p\int_\Omega  |\nabla(u-m)|^2\right)^{1/2}\left(\int_\Omega (\Delta(u-m))^2\right)^{1/2} \le \frac{1}{2}\int_\Omega |\nabla u |^2 \;dx + \frac{C_p}{2}\int_\Omega (\Delta u)^2
 \end{align*}
i.e.
 \begin{align*}
     \forall v \in H^2(\Omega) \text{ with } \partial_n v = 0\text{ on }\partial\Omega, \quad \int_\Omega |\nabla v |^2 \;dx &\le C_p\int_\Omega (\Delta v)^2\; dx.
 \end{align*}
We thus obtain, using the fact that for all $p\in \mathbb{N}$, $u_0^{p+1}\in H^2(\Omega)$ is such that $\partial_n u_0^{p+1} = 0$ (this is a consequence of the weak variational fomulation~\eqref{eq:w_to_u_weak0}):
 \begin{equation*}
    \begin{split}
    RE[\bbu^{(\tau)}(t) \, | \, \bbu^{\infty}]- RE[\bbu^0 \, | \, \bbu^{\infty}] & \leq - \frac{4k}{C_{\rm sob}} \int_0^t \into \sum_{i=0}^n u_i^{(\tau)} \ln \frac{u_i^{(\tau)}}{u_i^\infty} \, dx ds -  \frac{2k \eps}{{C_{\rm P}}} \int_0^t \|\nabla u_0^{(\tau)}\|_{L^2(\Om)}^2 \; ds  \\
   & \quad + 2k \beta \int_0^t \|\nabla u_0^{(\tau)}\|^2_{L^2(\Om)} + \|\sigma_\tau \nabla u_0^{(\tau)}\|^2_{L^2(\Om)} \; ds.
    \end{split}
\end{equation*}
We need to pass to the limit as $\tau \to 0^+$ before estimating further. On the one hand, using the analysis of~\cite{EMP2021}, it holds that $(u_0^{(\tau)})_{\tau >0}$ and $(\sigma_\tau u_0^{(\tau)})_{\tau >0}$ converge strongly to $u_0$ in $L^2((0,t);H^1(\Om))$, so we can pass to the limit in the three gradient terms. On the other hand, $(u_i^{(\tau)})_{\tau >0}$ converges strongly to $u_i$ in $L^2((0,t);L^2(\Om))$ for any $i=1,\dots,n$, so by the continuity of the function $[0,1] \ni x \mapsto x \ln x$ and Lebesgue dominated convergence theorem, we can pass to the limit in the mixing terms, and therefore in the entire energy itself as well. We obtain 
 \begin{equation*}
    \begin{aligned}
    RE[\bbu(t) \, | \, \bbu^{\infty}]- RE[\bbu^0 \, | \, \bbu^{\infty}]  &\leq - \frac{4k}{C_{\rm sob}} \int_0^t \into \sum_{i=0}^n u_i \ln \frac{u_i}{u_i^\infty} \, dx ds -  4k(\frac{\eps}{2C_p}-\beta) \int_0^t \|\nabla u_0\|_{L^2(\Om)}^2 \; ds \\
    &\leq - 4k \min \l(\frac{1}{C_{\rm sob}},\frac{1}{C_p}-\frac{2 \beta}{\eps} \r)\int_0^t RE[\bbu(t)\, | \, \bbu^{\infty}] \; ds,
    \end{aligned}
\end{equation*}
so that applying an integral version of Gronwall inequality gives \eqref{eq:conv-exp}-\eqref{eq:rate}.
Moreover, under the same condition, the Poincaré and the Csisz\'ar-Kullback inequalities (see \cite{jungel2016entropy}) show that the relative energy \eqref{rel-entr} dominates the square of the $L^1$-norm, hence the conclusion.
\end{proof}

\begin{remark}
    Note that the scope of the previous result is restricted to a regime where the system evolution is mainly diffusion-driven. In particular, the system does not enjoy phase separation. The dynamics outside of this regime is more relevant, but also much more difficult to study. Some constant states may become unstable, leading to the concepts of \emph{spinodal region} and \emph{spinodal decomposition} (see Figure~\ref{fig:spin}). This phenomenon has been studied in the context of the single-species Cahn-Hilliard equation \cite{novick2008cahn} and for Cahn-Hilliard systems \cite{eyre1993}. It is an interesting perspective to study it for cross-diffusion-Cahn-Hilliard systems, which we postpone to future work.
\end{remark}

\section{Finite volume scheme}
\label{sec:scheme}
We introduce a two-point finite volume scheme to solve \eqref{eq:system_intro} and study its large-time asymptotics. In the next subsection, we introduce a suitable discretization of the space-time domain and useful notations. Then we present our scheme and prove several properties related to the preservation of the structure of the continuous system. 

\subsection{Mesh and notations}
 An admissible mesh of $\Omega$ is a triplet $(\cT, \cE, (x_K)_{K\in \cT})$ such that the following conditions are fulfilled.
 \begin{itemize}
  \item [(i)] Each control volume (or cell) $K \in \cT$ is non-empty, open, polyhedral and convex. We assume that
  $$
  K \cap L = \emptyset \mbox{ if } K,L \in \cT \mbox{ with } K \neq L, \quad \mbox{ while }\bigcup_{K\in \cT} \overline{K} = \overline{\Omega}. 
  $$
  \item[(ii)] Each face $\sigma \in \cE$ is closed and is contained in a hyperplane of $\bR^d$, with positive $(d-1)$-dimensional Hausdorff (or Lebesgue) measure denoted by $m_\sigma= \cH^{d-1}(\sigma)>0$. We assume that 
  $\cH^{d-1}(\sigma \cap \sigma') = 0$ for $\sigma ,\sigma'\in \cE$ unless $\sigma  = \sigma'$. For all $K\in \cT$, we assume that there exists a subset $\cE_K$ of $\cE$ such that $\partial K = \bigcup_{\sigma \in \cE_K} \sigma$. Moreover, we suppose that 
  $\bigcup_{K\in \cT} \cE_K = \cE$. Given two distinct control volumes $K,L\in \cT$, the intersection $\overline{K} \cap \overline{L}$ either reduces to a single face $\sigma \in \cE$ denoted by $K|L$, or its $(d-1)$-dimensional Hausdorff measure is $0$.
  \item[(iii)] The cell-centers $(x_K)_{K\in \cT}$ satisfy $x_K \in K$, and are such that, if $K,L \in \cT$ share a face $K|L$, then the vector $x_L - x_K$ is orthogonal to $K|L$.
 \end{itemize}
We denote by $m_K$ the $d$-dimensional Lebesgue measure of the control volume $K$. 
The set of the faces is partitioned into two subsets: the set $\cE_{\rm int}$ of the interior faces defined by 
$$\cE_{\rm int} = \{ \sigma \in \cE \; | \; \sigma = K|L \mbox{ for some }K,L\in \cT\},$$ 
and the set $\cE_{\rm ext}=\cE \setminus \cE_{\rm int}$ of the exterior faces defined by 
$\cE_{\rm ext}=\{ \sigma \in \cE \; | \; \sigma \subset \partial \Omega \}$. 
For a given control volume $K\in \mathcal T$, we also define
$\cE_{K,{\rm int}} = \cE_K \cap \cE_{\rm int}$ (respectively $\cE_{K, {\rm ext}} = \cE_K \cap \cE_{\rm ext}$) 
the set of its faces that belong to $\cE_{\rm int}$ (respectively $\cE_{\rm ext}$). For such a face $\sigma \in \cE_{K,{\rm int}}$, we may write $\sigma = K|L$, meaning that $\sigma = \overline{K} \cap \overline{L}$, where $L\in \mathcal T$. 
Given $\sigma \in \cE$, we let 
$$
d_\sigma:= \left\{ 
\begin{array}{ll}
 |x_K - x_L| & \quad \mbox{ if } \sigma = K|L \in \cE_{\rm int},\\
 |x_K - x_\sigma| & \quad \mbox{ if } \sigma \in \cE_{K, {\rm ext}},\\
\end{array}
\right.
 \quad \mbox{ and } \quad \tau_\sigma = \frac{m_\sigma}{d_\sigma}.
$$
In what follows, we use boldface notations for any vector-valued quantity, typically for elements of  $\bR^{n+1}$, $\bR^{|\cT|}$, $\bR^{|\cE|}$, $(\bR^{|\cT|})^{n+1}$ and $(\bR^{|\cE|})^{n+1}$. Moreover, we use uppercase letters to denote discrete quantities, in contrast to lowercase letters used in the previous sections for functions. 
Given any discrete scalar field $\bbV = (V_K)_{K\in \cT} \in \bR^{|\cT|} $, we define for all cell $K\in \cT$ and interface $\sigma \in \cE_K$ the mirror value $V_{K\sigma}$ of $V_K$ across $\sigma$ by setting:
$$
V_{K\sigma} = \left\{
\begin{array}{ll}
 V_L &  \mbox{ if } \sigma = K|L \in \cE_{\rm int},\\
 V_K & \mbox{ if } \sigma \in \cE_{\rm ext}.\\
\end{array} 
\right.
$$
We also define the oriented and absolute jumps of $\bbV$ across any edge by
$$
D_{K\sigma} \bbV = V_{K\sigma} - V_K, \quad 
\mbox{ and } \quad D_\sigma \bbV = 
| D_{K\sigma} \bbV |, \quad \forall K\in \cT, \; \forall \sigma \in \cE_K.
$$
Note that in the above definition, for all $\sigma \in \cE$, the definition of 
$D_\sigma \bbV$ does not depend on the choice of the element $K\in \cT$ such that $\sigma \in \cE_K$. Therefore, it can be checked that the following discrete integration by parts formula holds, for any $\bbV, \bbW \in \bR^{|\cT|}$:
\begin{equation}
    \label{discrete-IBP}
    \sum_{K \in \cT} \sum_{\sigma \in \mathcal{E_{K}}} D_{K\sigma} \bbW. V_K = -\sum_{\sigma\in \mathcal E_{\rm int}} D_{K\sigma} \bbW. D_{K\sigma} \bbV.
\end{equation}
Concerning the time discretization of $(0,T)$, we consider $P_T\in \bN^*$ and an increasing infinite family of times $0<t_0 < t_1 < \cdots < t_{P_T} = T$ and we set $\Delta t_p=t_p - t_{p-1}$ for $p\in \{1, \cdots, P_T\}$.


\subsection{Numerical Scheme}
Let $\bbu^0=(u_0^0, \ldots, u_n^0)\in \cA_m$ be an initial condition satisfying the constraints \eqref{eq:constraint}. It is discretized on the mesh $\cT$ as 
\[
\bbU^0 =  \left( U^0_{i,K}\right)_{K\in \cT, 0\leq i \leq n}
\]
by setting 
\begin{equation}\label{eq:definit}
U^0_{i,K} = \frac{1}{m_K}\int_K u_i^0(x)\,dx, \quad \forall K\in \cT; ~ i=0,\dots,n.
\end{equation}
Assume that $\bbU^p = \left( U_{i,K}^p\right)_{K\in \cT, 0\leq i \leq n}$ is given for some $p \in \N$, then we have to define how to compute the discrete volume fractions $\bbU^{p+1} = \left( U_{i,K}^{p+1}\right)_{K\in \cT, 0\leq i \leq n}$. For $q=p,p+1$ and $i=0,\ldots,n$, we introduce the notation $\bbU^q_i:=\left(U^q_{i,K}\right)_{K\in \mathcal T}$. We also introduce discrete fluxes $\bbJ^{p+1}_{\cE} = \left(J_{i,K\sigma}^{p+1} \right)_{K\in \mathcal T, \sigma \in \cE_K, 0 \leq i \leq n}$, which are based on edge values $U^{p+1}_{i,\sigma}$ for all $\sigma \in \cE,i=0,\dots,n$. For any $K\in \cT$ such that $\sigma \in \cE_K$, the definition of $U^{p+1}_{i,\sigma}$ makes use of the values
$U^{p+1}_{i,K}$ and $U^{p+1}_{i,K\sigma}$ but is independent of the choice of $K$. The edge volume fractions are then defined through a logarithmic mean as follows
\begin{subequations}
\label{eq:scheme}
\begin{equation}\label{eq:u_isig}
U^{p+1}_{i,\sigma} = \left\{
\begin{array}{ll}
 0 & \mbox{ if } \min(U_{i,K}^{p+1}, U_{i, K\sigma}^{p+1})\leq 0, \\
 U_{i,K}^{p+1} & \mbox{ if } 0 < U_{i,K}^{p+1} =U_{i, K\sigma}^{p+1}, \\
 \frac{U_{i,K}^{p+1} - U_{i,K\sigma}^{p+1}}{\ln(U_{i,K}^{p+1}) - \ln(U_{i,K\sigma}^{p+1})} & \mbox{ otherwise}.\\
\end{array}
\right.
\end{equation}
This choice is motivated by a discrete chain rule property: for any $i=0,\dots,n$, $K \in \cT$, $\sigma \in \cE_K$, if $U_{i,K}^{p+1}, U_{i,K\sigma}^{p+1} > 0$ then 
\begin{equation}
    \label{discrete-chain-rule}
    D_{K\sigma} \bbU_i^{p+1} = U_{i,\sigma}^{p+1} D_{K\sigma} \ln(\bbU_i^{p+1}).
\end{equation}
We employ a time discretization relying on the backward Euler scheme:
\begin{equation}\label{eq:cons}
m_K \frac{U_{i,K}^{p+1} - U_{i,K}^p}{\Delta t_p} + \sum_{\sigma \in \cE_{K,\rm int}} J_{i, K\sigma}^{p+1} = 0, \quad \forall K\in \cT, \; i=0,\dots,n,
\end{equation}
and the discrete fluxes are adapted from formulas \eqref{eqi}-\eqref{eq0} as
\begin{equation}
    \label{eq:flux-discr}
    \begin{aligned}
    J_{i,K\sigma}^{p+1} &=   -\tau_\sigma \sum_{0 \leq j\neq i \leq n} K_{ij} \left( U^{p+1}_{j,\sigma} D_{K\sigma} \bbU_i^{p+1}  - U^{p+1}_{i,\sigma} D_{K\sigma} \bbU_j^{p+1} \right)   +  \tau_\sigma K_{i0} U^{p+1}_{i,\sigma}U^{p+1}_{0,\sigma} D_{K\sigma} \bbW_0^{p+\frac{1}{2}}, ~  i=1,\dots,n, \\
  J_{0,K\sigma}^{p+1} &= - \tau_\sigma \sum_{i=1}^n  K_{i0} \left( U^{p+1}_{i,\sigma} D_{K\sigma} \bbU_0^{p+1}  - U^{p+1}_{0,\sigma} D_{K\sigma} \bbU_i^{p+1} \right)  - \tau_\sigma \sum_{i=1}^n K_{i0} U^{p+1}_{i,\sigma}U^{p+1}_{0,\sigma} D_{K\sigma} \bbW_0^{p+\frac{1}{2}},
  \end{aligned}
\end{equation}
where the auxiliary variable $w_0$ is discretized from \eqref{eq:w0} as $\bbW_0^{p+1/2} = \left(W_{0,K}^{p+1/2}\right)_{K\in \mathcal T}$, where for any $K \in \cT$,
\begin{equation}
\label{eq:w0-discrete}
W_{0,K}^{p+\frac{1}{2}} =  - \frac{\eps}{m_K} \sum_{\sigma \in \mathcal E_{K,\rm int}} \tau_\sigma D_{K\sigma} \bbU_0^{p+1} + \beta (1 - 2U_{0,K}^p).
\end{equation}
\end{subequations}
Note that, in the latter formula, we apply the same convex-concave splitting as in \cite{eyre1993}: the convex part of the energy is discretized implicitly while the concave part is discretized explicitly. This well-known technique is crucial in order to recover free energy dissipation at the discrete level \cite{eyre1998unconditionally}, see the proof of Proposition~\ref{prop:energy-dissip}.

\begin{remark}
Rather than defining the system \eqref{eq:cons} for all $i=0,\dots,n$ and then check that it holds $\sum_{i=0}^n U_{i,K}^{p+1} = 1$ for any $K \in \cT$, $p \in \N$ , one could have written only $n$ equations and define $U_{0,K}^{p+1} = 1 - \sum_{i=1}^n U_{i,K}^{p+1}$ (it is the chosen approach to define weak continuous solutions). However, it is important \emph{not to} define the value on the edges $U_{0,\sigma}^{p+1}$ as $1 - \sum_{i=1}^n U_{i,\sigma}^{p+1}$ but rather as the logarithmic mean according to \eqref{eq:u_isig}. In doing so, the saturation constraint does not necessarily hold on the edges anymore, but one can recover free energy dissipation.
\end{remark}

\subsection{Elements of numerical analysis}
\begin{lemma}[Mass conservation]
For all $p \in \N$, it holds
\begin{equation*}
    \sum_{K \in \cT} m_K U_{i,K}^{p+1} = \int_{\Omega} u_i^0(x) dx, ~ i=0,\dots,n.
\end{equation*}
\end{lemma}

\begin{proof}
Summing \eqref{eq:cons} over $K \in \cT$ leads to:
\[\sum_{K \in \cT} m_K (U_{i,K}^{p+1} - U_{i,K}^p) = -\Delta t_p \sum_{K \in \cT} \sum_{\sigma \in \cE_{K,\rm int}} J_{i,K\sigma}^{p+1}, ~ i=0,\dots,n, \]
and this quantity is null because of cancellations on both sides of $\sigma \in \cE_{\rm int}$. The conclusion follows by induction and using \eqref{eq:definit}.
\end{proof}

\begin{lemma}[Volume-filling constraint]
Let $p \in \N$ and assume that $\bbU^p = \left( U_{i,K}^p\right)_{K\in \cT, 0\leq i \leq n}$ satisfies the volume-filling constraint $\sum_{j=0}^n U_{j,K}^p = 1$, for any $K \in \cT$. Then any solution $\bbU^{p+1}$ to the scheme \eqref{eq:scheme} satisfies it as well.
\end{lemma}
\begin{proof}
Fix $K \in \cT$ and sum \eqref{eq:cons} over $i=0,\dots,n$. Then on the one hand we have the sum of the cross-diffusion contributions
\[  \tau_\sigma \sum_{i=0}^n \sum_{j=0}^n K_{ij} \left( U^{p+1}_{j,\sigma} D_{K\sigma} \bbU_i^{p+1}  - U^{p+1}_{i,\sigma} D_{K\sigma} \bbU_j^{p+1} \right)\] 
that cancels thanks to the symmetry of the coefficients $K_{ij}$. On the other hand, it is clear by construction that the Cahn-Hilliard term in $J_{0,K\sigma}^{p+1}$ exactly compensates the sum of the Cahn-Hilliard terms in the $n$ other fluxes. Therefore, one obtains
\[
m_K \frac{1}{\Delta t_p} \sum_{i=0}^n \left(U_{i,K}^{p+1} - U_{i,K}^p \right) = 0.
\]
\end{proof}

\begin{lemma}[Weak positivity property]
Let $p \in \N$ and assume that $\bbU^p$ is nonnegative. Then any solution $\bbU^{p+1}$ to the scheme is nonnegative. If moreover $\bbU^p$ is positive, then any solution $\bbU^{p+1}$ is positive as well. 
\end{lemma}

\begin{proof}
We begin by proving nonnegativity. Fix $i \in \{0,\dots,n\}$. Reason by contradiction and assume that $\bbU^{p+1}_i$ has a (strictly) negative minimum $U_{i,K}^{p+1}$ in the cell $K \in \cT$. The conservation scheme \eqref{eq:cons} in the cell $K$ gives that
\[ m_K \frac{U_{i,K}^{p+1} - U_{i,K}^p}{\Delta t_p} = -\sum_{\sigma \in \cE_{K,\rm int}} J_{i,K\sigma}^{p+1},  \]
and using that $\bbU^p_i$ is nonnegative we get that
\begin{equation}
    \label{eq:strict}
    \sum_{\sigma \in \cE_{K,\rm int}} J_{i,K\sigma}^{p+1} > 0.
\end{equation}  
But using that, for any $\sigma \in \cE_{K,\rm int}$, $U_{i,\sigma}^{p+1}=0$ by definition \eqref{eq:u_isig}, several terms cancel in the fluxes formula \eqref{eq:flux-discr}, and we get 
\[ \sum_{\sigma \in \cE_{K,\rm int}} \left( -\tau_\sigma \sum_{0 \leq j \neq i \leq n} K_{ij} U_{j,\sigma}^{p+1} D_{K\sigma} \bbU_i^{p+1} \right) > 0,  \]
but since $K_{ij},U_{j,\sigma}^{p+1}, D_{K\sigma} \bbU_i^{p+1},\tau_\sigma \geq 0$, we obtain a contradiction. Therefore $\bbU^{p+1}_i \geq 0$. \newline
If $\bbU_i^p$ is furthermore (strictly) positive then we can assume the minimum $U_{i,K}^{p+1}$ to be only nonnegative and still obtain \eqref{eq:strict}, so the same reasoning gives $\bbU^{p+1}_i > 0$.
\end{proof}
Thanks to the previous lemma, the chain rule \eqref{discrete-chain-rule} is always valid provided the initial condition is positive everywhere. We make this assumption in the following. As a consequence, we can define the discrete chemical potentials (see \eqref{eq:mui}-\eqref{eq:mu0}) as follows: for any $K \in \cT$, 
\begin{equation}
    \label{eq:mui-discr}
    \mu_{i,K}^{p+1} = \ln(U_{i,K}^{p+1}), ~ i=1,\dots,n,
\end{equation}
and
\begin{equation}
    \label{eq:mu0-discr}
    \mu_{0,K}^{p+1} = \ln(U_{0,K}^{p+1}) + W_{0,K}^{p+\frac{1}{2}}.
\end{equation}
For all $i=0,\ldots,n$, we introduce $\bbmu_i^{p+1}:=(\mu^{p+1}_{i,K})_{K\in \mathcal T}$. The discrete fluxes \eqref{eq:flux-discr} can thus be rewritten in the following entropic form (independently of the discretization formula \eqref{eq:w0-discrete}):
\begin{equation}
    \label{eq:flux-discr-entropic}
    \begin{aligned}
    J_{i,K\sigma}^{p+1} &=  - \tau_\sigma \sum_{0 \leq j \neq i \leq n} K_{ij} U_{i,\sigma}^{p+1} U_{j,\sigma}^{p+1}  D_{K \sigma} \left(\bbmu_i^{p+1} - \bbmu_j^{p+1} \right), ~ i=0,\dots,n.
  \end{aligned}
\end{equation}
This latter formulation of the fluxes is at the core of the discrete free energy dissipation inequality. Let us define, according to \eqref{eq:energy}, the following discrete free energy functionals: for all $\bbV = (V_{i,K})_{0\leq i \leq n,K\in \mathcal T}$,
\begin{equation}
\label{eq:energy-discr}
\begin{aligned}
E_{\cT}(\bbV) &:= E_{\rm conv,\cT}(\bbV) + E_{\rm conc,\cT}(\bbV), \\
E_{\rm conv, \cT}(\bbV) &:= \sum_{i=0}^n \sum_{K\in\cT} m_K (V_{i,K}\ln(V_{i,K}) - V_{i,K} + 1) + \frac{\eps}{2} \sum_{\sigma\in \mathcal E_{\rm int}, \; \sigma = K|L} \tau_\sigma |D_{K\sigma} \bbV_{0}|^2, \\
E_{\rm conc, \cT}(\bbV) &:= \beta \sum_{K\in\cT} m_K V_{0,K}(1- V_{0,K}),
\end{aligned}
\end{equation}
with $\bbV_0=(V_{0,K})_{K\in \mathcal T}$. Their differentials at $\bbU = (U_{i,K})_{0\leq i \leq n, K\in \mathcal T}$ are given by
\begin{equation}
    \label{eq:energy-deriv}
    \begin{aligned}
    DE_{\rm conv, \cT}(\bbU)\cdot \bbV &= \sum_{i=0}^n \sum_{K \in \cT} m_K \ln(U_{i,K})V_{i,K} + \eps \sum_{\sigma\in \mathcal E_{\rm int}, \; \sigma = K|L}  \tau_{\sigma} D_{K\sigma} \bbU_0 D_{K\sigma} \bbV_0, \\
    DE_{\rm conc, \cT}(\bbU)\cdot \bbV &= \beta \sum_{K\in \cT} m_K (1-2U_{0,K})V_{0,K}.
    \end{aligned}
\end{equation}
Remark that, by convexity (resp. concavity), it holds, for $p \in \N$:
\begin{equation}
    \label{eq:conv-conc}
\begin{aligned}
    E_{\rm conv, \cT}(\bbU^{p+1}) - E_{\rm conv, \cT}(\bbU^p) & \leq DE_{\rm conv, \cT}(\bbU^{p+1})\cdot(\bbU^{p+1}-\bbU^p), \\
    E_{\rm conc, \cT}(\bbU^{p+1}) - E_{\rm conc, \cT}(\bbU^p) & \leq DE_{\rm conc, \cT}(\bbU^p)\cdot(\bbU^{p+1}-\bbU^p).
\end{aligned}
\end{equation}
We establish a discrete free energy dissipation inequality, as stated in the next proposition. Recall the definition \eqref{eq:mobility} of the mobility matrix. 
\begin{e-proposition}
\label{prop:energy-dissip}
Let $p \in \N$ and $\bbU^p$ be (strictly) positive. Then any solution $\bbU^{p+1}$ to the scheme satisfies 
\begin{equation}
    \label{eq:energy-dissip-discr}
     E_{\cT}(\bbU^{p+1}) - E_{\cT}(\bbU^p)  + \Delta t_p \sum_{\sigma \in \mathcal{E}_{\rm int}} \tau_\sigma  (D_{K\sigma} \bbmu^{p+1})^T M(\bbU_{\sigma}^{p+1}) D_{K\sigma} \bbmu^{p+1}  \leq 0,
\end{equation}
where $\bbU_\sigma^{p+1}= \left( U_{i,\sigma}^{p+1}\right)_{0\leq i \leq n}$. In particular, since $M(\bbU_\sigma^{p+1})$ is always a positive semi-definite matrix, it holds
\[E_{\cT}(\bbU^{p+1}) \leq E_{\cT}(\bbU^p) \leq E_{\cT}(\bbU^0). \]
\end{e-proposition}

\begin{proof}
Multiply equations \eqref{eq:cons} by $\Delta t_p \mu_{i,K}^{p+1}$ and sum over all species $i=0,...,n$ and all cells $K \in \cT$:
\begin{equation}
\label{eq:bilan}
    \sum_{i=0}^n \sum_{K \in \cT} m_K \left(U_{i,K}^{p+1}-U_{i,K}^p\right) \mu_{i,K}^{p+1} = - \Delta t_p \sum_{i=0}^n \sum_{K \in \cT} \sum_{\sigma \in \cE_K} J_{i,K\sigma}^{p+1} \mu_{i,K}^{p+1}. 
\end{equation} 
On the left-hand side, we obtain using the definitions \eqref{eq:mui-discr}-\eqref{eq:mu0-discr} of the discrete chemical potentials:
\begin{equation*} 
\sum_{i=0}^n \sum_{K \in \cT} m_K \left(U_{i,K}^{p+1}-U_{i,K}^p\right) \mu_{i,K}^{p+1} = \sum_{K\in \cT} \sum_{i=0}^n m_K \left(U_{i,K}^{p+1} - U_{i,K}^p\right) \ln U_{i,K}^{p+1} + \sum_{K\in \cT}m_K  \left(U_{0,K}^{p+1} - U_{0,K}^p\right) W_{0,K}^{p+\frac{1}{2}},
\end{equation*}
where we obtain two different contributions that we want to identify to derivatives of the discrete energies \eqref{eq:energy-deriv}. The first term identifies to the Boltzmann part, taken at $\bbU^{p+1}$, against $(\bbU^{p+1}-\bbU^p)$. Using the definition \eqref{eq:w0-discrete} of $W_{0,K}^{p+\frac{1}{2}}$ and using the discrete integration by part formula \eqref{discrete-IBP}, the second term can be expressed as: 
\begin{align*}
&\sum_{K\in \cT}m_K  \left(U_{0,K}^{p+1} - U_{0,K}^p\right) W_{0,K}^{p+\frac{1}{2}}\\
&=   -\eps \sum_{K\in \cT}  (U_{0,K}^{p+1} - U_{0,K}^p) \sum_{\sigma \in \mathcal E_{K,\rm int}} \tau_\sigma D_{K\sigma} \bbU_0^{p+1}  + \beta \sum_{K\in \cT}m_K (U_{0,K}^{p+1} - U_{0,K}^p) (1 - 2U_{0,K}^p), \\
&= \eps \sum_{\sigma \in \mathcal E_{\rm int}} \tau_{\sigma} D_{K\sigma} \bbU_0^{p+1} D_{K \sigma} \left( \bbU_0^{p+1} - \bbU_0^p \right) + \beta \sum_{K\in \cT}m_K (U_{0,K}^{p+1} - U_{0,K}^p) (1 - 2U_{0,K}^p),
\end{align*}
where we identified the Cahn-Hilliard contributions of \eqref{eq:energy-deriv}, respectively taken at $\bbU^{p+1}$ and $\bbU^p$, against $\bbU^{p+1}-\bbU^{p}$. Putting everything together, we can identify the total derivative of the energy:
\[ \sum_{i=0}^n \sum_{K \in \cT} m_K \left(U_{i,K}^{p+1}-U_{i,K}^p\right) \mu_{i,K}^{p+1} = DE_{\rm conv, \cT}(\bbU^{p+1})\cdot \left(\bbU^{p+1}-\bbU^p \right) + DE_{\rm conc, \cT}(\bbU^p)\cdot \left(\bbU^{p+1} - \bbU^p \right),  \]
and we conclude from \eqref{eq:bilan}, using the convexity (resp. concavity) inequalities \eqref{eq:conv-conc}, that:
\[ E_{\cT}(\bbU^{p+1}) - E_{\cT}(\bbU^p) \leq - \Delta t_p \sum_{i=0}^n \sum_{K \in \cT} \sum_{\sigma \in \cE_K} J_{i,K\sigma}^{p+1} \mu_{i,K}^{p+1} . \]
On the other hand, using the entropic form of the fluxes \eqref{eq:flux-discr-entropic}, the right-hand side reads: 
\[- \sum_{i=0}^n \sum_{K \in \cT} \sum_{\sigma \in \cE_K} J_{i,K\sigma}^{p+1} \mu_{i,K}^{p+1} = \sum_{i=0}^n \sum_{K \in \cT}  \sum_{\sigma \in \mathcal{E}_{K,\rm int}} \tau_\sigma \left(\sum_{0 \leq j \neq i \leq n} K_{ij} U_{i,\sigma}^{p+1} U_{j,\sigma}^{p+1} D_{K \sigma} (\bbmu_j^{p+1} - \bbmu_i^{p+1}) \right) \mu_{i,K}^{p+1}.  \]
Using once again integration by parts, this is equal to
\[-\sum_{\sigma \in \mathcal{E}_{\rm int}} \tau_\sigma \sum_{i=0}^n \sum_{0 \leq j \neq i \leq n} K_{ij} U_{i,\sigma}^{p+1} U_{j,\sigma}^{p+1} D_{K \sigma} (\bbmu_i^{p+1} - \bbmu_j^{p+1}) D_{K\sigma} \bbmu_i^{p+1}, \]
and having in mind the expression of the mobility matrix \eqref{eq:mobility}, we finally obtain
\begin{equation*}
    - \sum_{i=0}^n \sum_{K \in \cT} \sum_{\sigma \in \cE_K} J_{i,K\sigma}^{p+1} \mu_{i,K}^{p+1} = -\sum_{\sigma \in \mathcal{E}_{\rm int}} \tau_\sigma (D_{K\sigma} \bbmu^{p+1})^T M(\bbU_{\sigma}^{p+1}) D_{K\sigma} \bbmu^{p+1}.
    \end{equation*}
We have obtained \eqref{eq:energy-dissip-discr}.
\end{proof}

\section{Numerical Simulations}
\label{sec:simulations}
The numerical scheme has been implemented in the Julia language. At each time step, the nonlinear system is solved using Newton's method with stopping criterion $\|\bbU^{p,k+1}-\bbU^{p,k} \|_{\infty} \leq 10^{-10}$ and adaptive time stepping. We always consider the case $n=2$ of three species and the cross-diffusion coefficients are chosen to be $K_{01}=K_{10}= 0.2, ~ K_{12}=K_{21}=0.1, ~ K_{02}=K_{20}=1$ (diagonal coefficients do not play any role). We study the dynamics for different values of $\eps$ and $\beta$ and different initial conditions, with a particular focus on the stationary solutions obtained in the long-time asymptotics and the shape of the free energy over time. 

\subsection{One-dimensional simulations}
We consider the domain $(0,1)$, a uniform mesh of $100$ cells, a maximal time step $\Delta t_1 = 10^{-3}$ and three initial profiles defined as smooth perturbations of a constant state: for $x \in (0,1)$,
\begin{equation*}
u^0_{0}(x) = u^0_{1}(x) = \frac{1}{4} \l(1 + \kappa \cos(k\pi x)\r), ~  u_{2}^0(x) = 1-u_0^0(x)-u_1^0(x), 
\end{equation*}
where the perturbation is parametrized by its amplitude $\kappa \in (0,1)$ and frequency $k \in \N^*$, making sure that the box constraints in \eqref{eq:constraint} are respected. Note that the mass is preserved by the perturbation, so that it holds $m_0 = \into u_0^0 \, dx  = 0.25$. Moreover, the Poincaré constant of the domain is $C_p=1$.

\medskip

We begin by illustrating Theorem~\ref{thm:convergence} and we denote by $\bbU^{\infty}$ the discrete constant state defined according to \eqref{eq:u_const}. We first consider $\eps=4, \beta=1$ so that condition \eqref{eq:cond-global-stab} is satisfied. Starting from various initial conditions, we always observe a diffusive behaviour with exponential convergence to $\bbU^{\infty}$, which confirms the globally stable behaviour. In Figure~\ref{fig:conv}, we give some snapshots of the evolution and plot the discrete relative free energy $RE_{\cT}[\bbU^p \, | \, \bbU^{\infty}]= E_{\cT}(\bbU^p)- E_{\cT}(\bbU^{\infty})$ (see \eqref{eq:energy-discr}) over the discrete time, starting from a specific initial condition and with time horizon $T_1=10$. For $\eps=0.5,\beta=2$, \eqref{eq:cond-global-stab} is not anymore satisfied, but the simulations behave similarly. Note that the convexity condition \eqref{eq:convexity-cond} on the energy reads
\begin{equation*}
    2 + \frac{\eps}{2} - \beta \geq 0,
\end{equation*}
and is verified in this case. It seems that this condition is more determinant for the dynamics than \eqref{eq:cond-global-stab}.
\begin{figure*}
  \centering
  \begin{subfigure}[b]{0.35\textwidth}
      \centering
      \includegraphics[width=\textwidth]{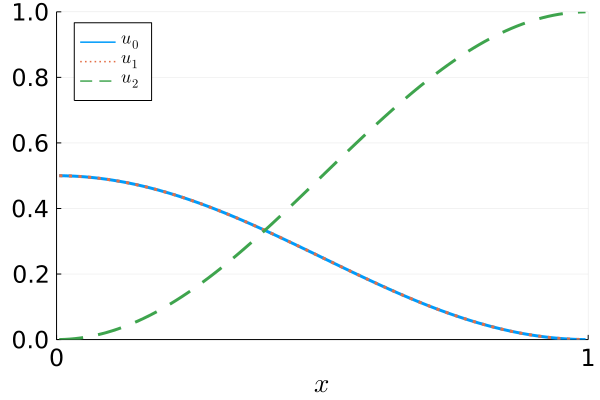}
      \caption{\small $t=0$} 
  \end{subfigure}
  \quad
  \begin{subfigure}[b]{0.35\textwidth}  
      \centering 
      \includegraphics[width=\textwidth]{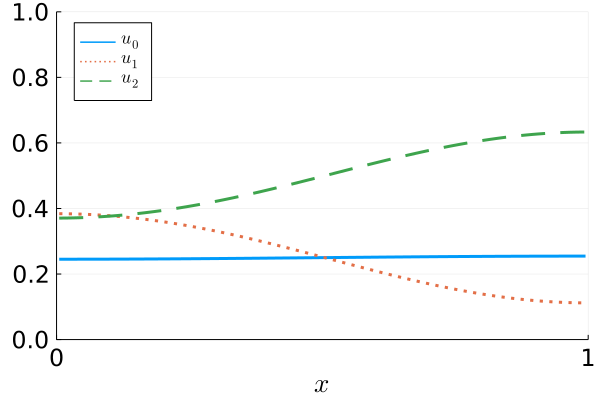}
      \caption{\small $t=0.5$}    
  \end{subfigure}
  
  \begin{subfigure}[b]{0.35\textwidth}   
      \centering 
      \includegraphics[width=\textwidth]{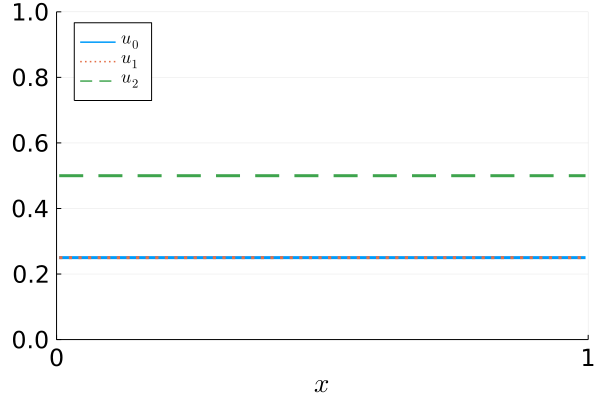}
      \caption{\small $T_1=10$}  
  \end{subfigure}
  \quad
  \begin{subfigure}[b]{0.35\textwidth}   
      \centering 
      \includegraphics[width=\textwidth]{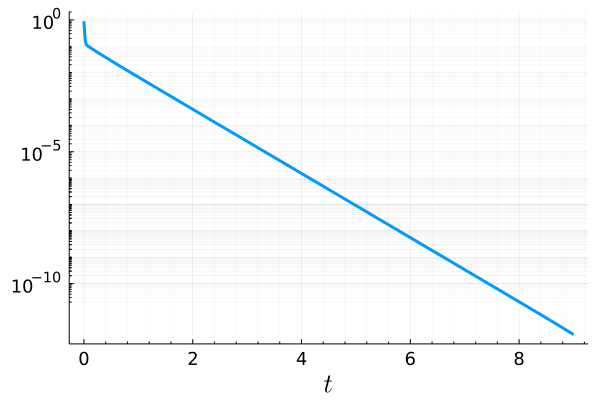}
      \caption{\small $RE[\bbu(t) \, | \, \bbu^{\infty}]= E(\bbu(t))- E(\bbu^{\infty})$}  
  \end{subfigure}
  \caption[]
  {\small Profiles along time in the globally stable case. $\eps=4,\beta=1,\kappa=k=1$. } 
  \label{fig:conv}
\end{figure*}

We are now interested in the unstable regime and consider $\eps=0.1,\beta=10$, so that both \eqref{eq:cond-global-stab} and \eqref{eq:convexity-cond} are widely violated. In Figure~\ref{fig:non-conv1}, we display the results of the dynamics starting from the same initial condition as before and with a time horizon $T_2=8$. We observe convergence to a non-constant stationary solution with a diffuse segregation interface and exponential decrease to 0 of the quantity $E_{\cT}(\bbU^p)- E_{\cT}(\bbU^{P_{T_2}})$ over the discrete time. Let us make two remarks about this quantity: first, in contrast to the stable situation, we can only measure the difference with respect to the final solution $\bbU^{P_{T_2}}$, since we do not know \emph{a priori} the limit of the dynamics. Second, since the limit is not homogeneous in space, the quantity $E_{\cT}(\bbU^p)- E_{\cT}(\bbU^{P_{T_2}})$ differs by a linear term from an approximation of the relative energy \eqref{rel-entr}. In Figure~\ref{fig:non-conv2}, we display the results of the dynamics starting from an initial condition with higher frequency $k=2$ over a time horizon $T_3=2$. We observe exponential convergence to another stationary solution, for which the free energy is smaller. Note that, in accordance with the results of Theorem~\ref{adm-min}, diffusion prevents the profiles from being exactly zero somewhere, and we always observe a $\delta > 0$ such that all profiles are uniformly bounded between $\delta$ and $1-\delta$.

\begin{figure*}
  \centering
  \begin{subfigure}[b]{0.35\textwidth}
      \centering
      \includegraphics[width=\textwidth]{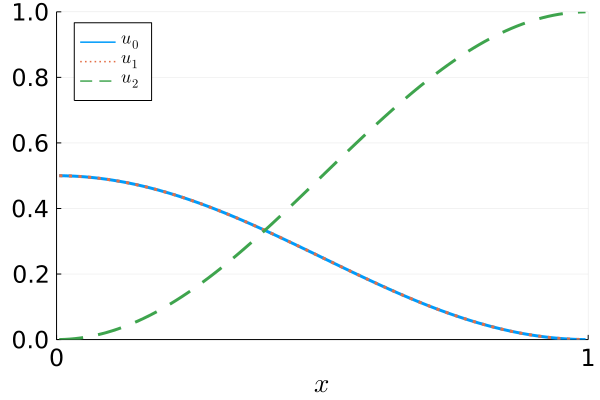}
      \caption{\small $t=0$}    
  \end{subfigure}
  \quad
  \begin{subfigure}[b]{0.35\textwidth}  
      \centering 
      \includegraphics[width=\textwidth]{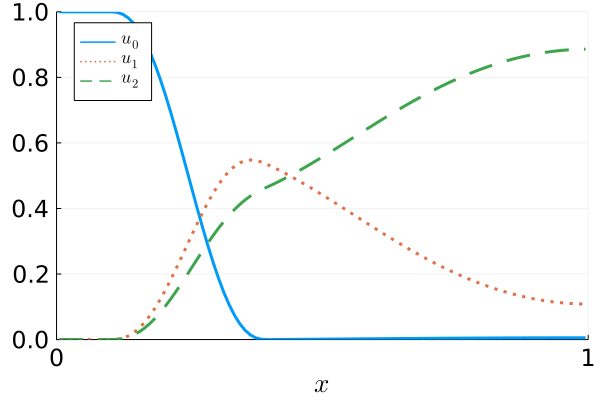}
      \caption[]%
      {{\small $t=0.4$}}    
  \end{subfigure}

  \begin{subfigure}[b]{0.35\textwidth}   
      \centering 
      \includegraphics[width=\textwidth]{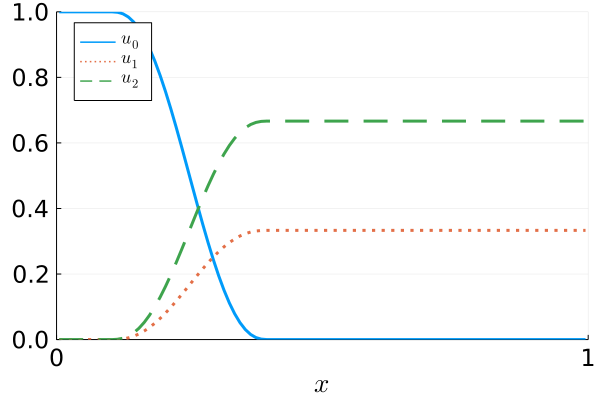}
      \caption[]
      {{\small $T_2=8$}}    
  \end{subfigure}
  \quad
  \begin{subfigure}[b]{0.35\textwidth}   
      \centering 
      \includegraphics[width=\textwidth]{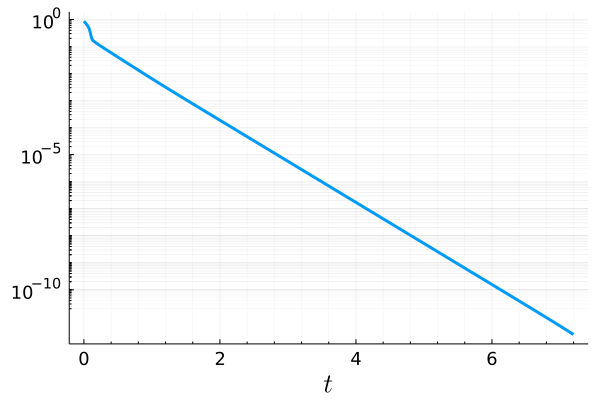}
      \caption[]%
      {{\small $E(\bbu(t))- E(\bbu(T_2))$}}  
  \end{subfigure}
  \caption[]
  {\small Profiles along time in the non-convex case. $\eps=0.1,\beta=10,\kappa=k=1$. } 
  \label{fig:non-conv1}
\end{figure*}

\begin{figure*}
  \centering
  \begin{subfigure}[b]{0.35\textwidth}
      \centering
      \includegraphics[width=\textwidth]{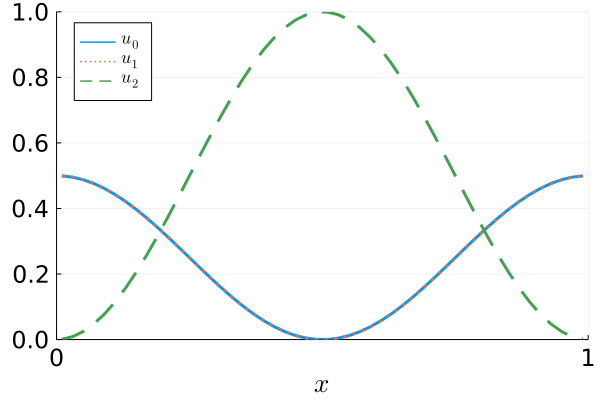}
      \caption[Network2]%
      {{\small $t=0$}}    
  \end{subfigure}
  \quad
  \begin{subfigure}[b]{0.35\textwidth}  
      \centering 
      \includegraphics[width=\textwidth]{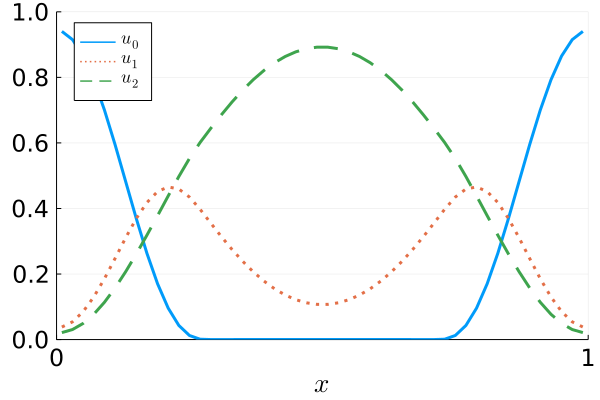}
      \caption[]%
      {{\small $t=0.01$}}    
  \end{subfigure}

  \begin{subfigure}[b]{0.35\textwidth}   
      \centering 
      \includegraphics[width=\textwidth]{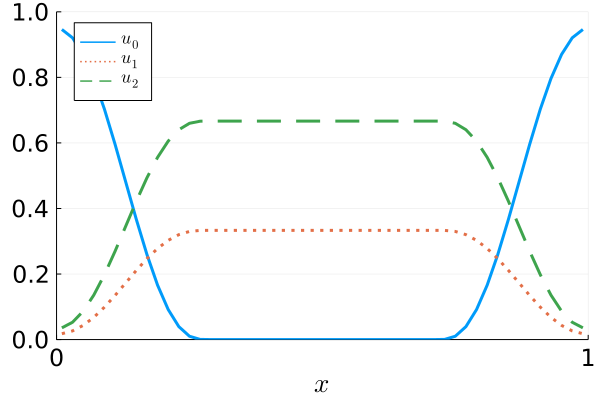}
      \caption[]
      {{\small $T_3=2$}}    
  \end{subfigure}
  \quad
  \begin{subfigure}[b]{0.35\textwidth}   
      \centering 
      \includegraphics[width=\textwidth]{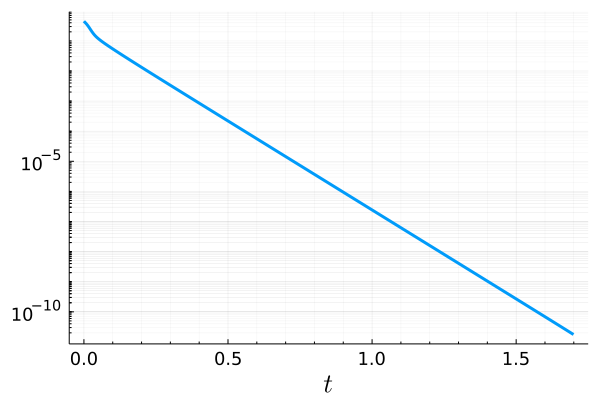 }
      \caption[]%
      {{\small $E(\bbu(t))- E(\bbu(T_3))$}}  
  \end{subfigure}
  \caption[]
  {\small Profiles along time in the non-convex case. $\eps=0.1, \beta=10,\kappa=1, k=2$. } 
  \label{fig:non-conv2}
\end{figure*}
Finally, we provide numerical evidence that, in the general regime, the solution should converge to a local minimizer of the energy. Although the property of being a local minimizer cannot be easily verified numerically, one can use Theorem~\ref{thm:optimality} and verify that the obtained numerical stationary solutions satisfy a discrete version of the optimality conditions \eqref{eq:u0}-\eqref{eq:ui}. In Figure~\ref{fig:crit-res}, we plot, for the two previous simulations, the $l^{\infty}$ norm of the residual of this system over time and observe exponential convergence to 0, which indicates that the solution converges to a critical point of the energy. 

\begin{figure*}
  \centering
  \begin{subfigure}[b]{0.35\textwidth}
      \centering
      \includegraphics[width=\textwidth]{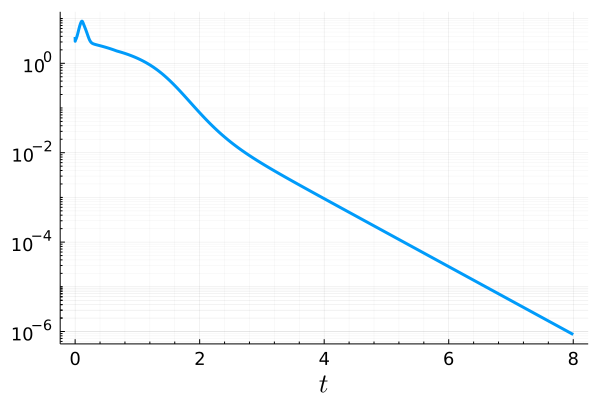}
  \end{subfigure}
  \quad
  \begin{subfigure}[b]{0.35\textwidth}  
      \centering 
      \includegraphics[width=\textwidth]{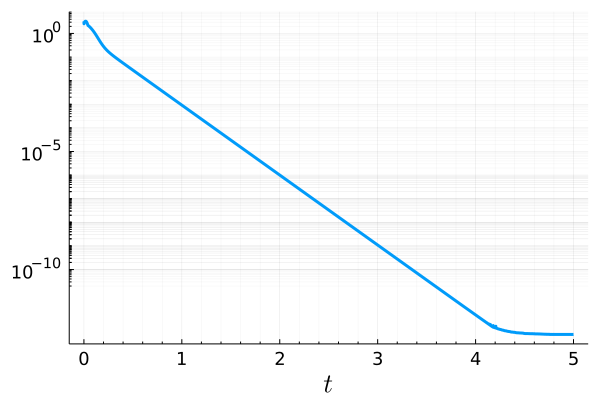}
  \end{subfigure}
  \caption[]
  {\small Residual of the Euler-Lagrange system in the case $k=1$ (left) and $k=2$ (right).} 
  \label{fig:crit-res}
\end{figure*}

\FloatBarrier

\subsection{Two-dimensional simulations}
We consider the domain $(0,1)^2$, a uniform mesh of $150^2$ squares and a maximal time step $\Delta t_2 = 5 \times 10^{-3}$. We want to observe spinodal decomposition, so we pick initial profiles defined by random perturbation of constant states: for $(x,y) \in (0,1)^2$,
\begin{equation*}
\begin{aligned}
    u^0_{0}(x,y) &=  0.5 + 2 \kappa \l(\eta_0(x,y) - \frac{1}{2}\r), \\
    u^0_{1}(x,y) &= 0.4 + 2 \kappa \l(\eta_1(x,y) - \frac{1}{2}\r), \\
    u^0_2(x,y) &= 1 - u^0_0(x,y) - u_1^0(x,y),
\end{aligned}
\end{equation*}
where, for any $(x,y) \in (0,1)^2$, $\eta_0(x,y)$ and $\eta_1(x,y)$ are independent noises drawn uniformly in $(0,1)$ and $\kappa=10^{-2}$. We choose the parameters $\eps = 10^{-3}, \beta=5$. The results of the simulation are given in Figure~\ref{fig:spin}. As expected, $u_0$ quickly separates from the two other species. Then on a slower time scale, the effect of cross-diffusion homogenizes $u_1$ and $u_2$ to the constant states. Finally, the coarsening process happens on a much slower time scale, minimizing the interface energy.

\begin{figure}
    \centering
    \begin{subfigure}[b]{0.3\textwidth}
        \includegraphics[width=\textwidth]{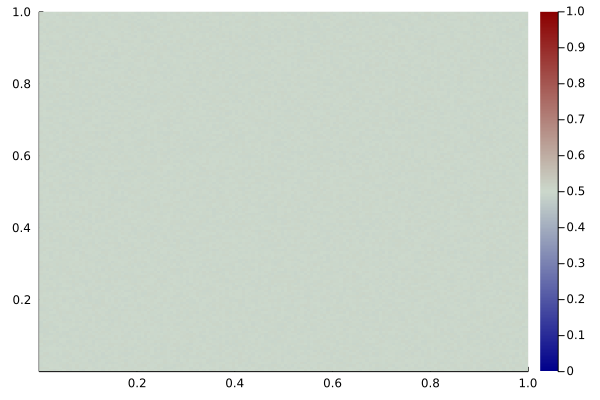}
        \caption{$u_0$}
    \end{subfigure}
    ~ 
    \begin{subfigure}[b]{0.3\textwidth}
        \includegraphics[width=\textwidth]{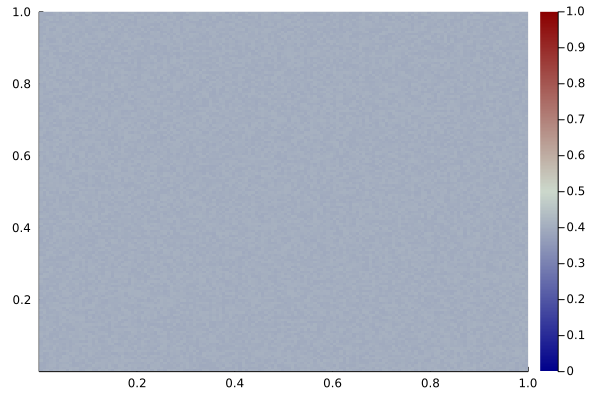} 
        \caption{$u_1$}
    \end{subfigure}
    ~ 
    \begin{subfigure}[b]{0.3\textwidth}
        \includegraphics[width=\textwidth]{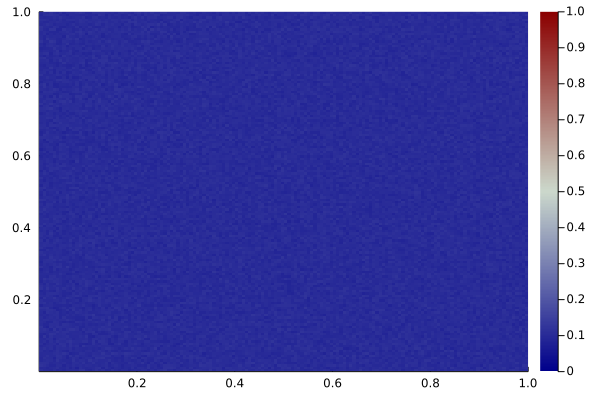}
        \caption{$u_2$}
    \end{subfigure}

    \begin{subfigure}[b]{0.3\textwidth}
        \includegraphics[width=\textwidth]{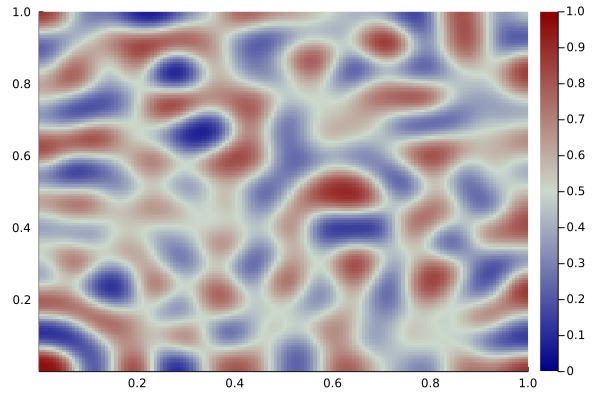}
        \caption{$u_0$}
    \end{subfigure}
    ~ 
    \begin{subfigure}[b]{0.3\textwidth}
        \includegraphics[width=\textwidth]{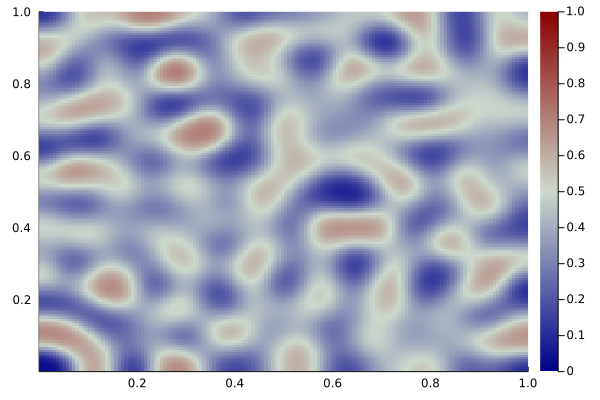} 
        \caption{$u_1$}
    \end{subfigure}
    ~ 
    \begin{subfigure}[b]{0.3\textwidth}
        \includegraphics[width=\textwidth]{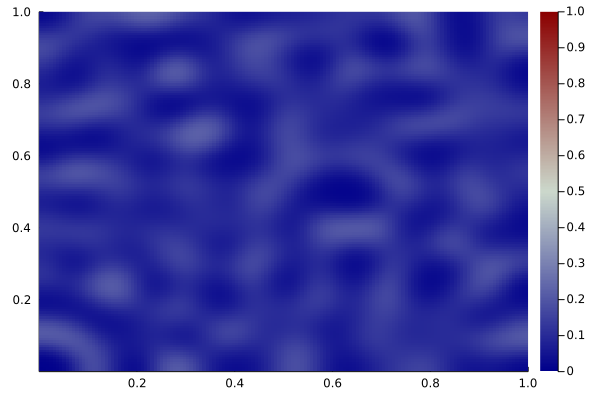}
        \caption{$u_2$}
    \end{subfigure}

    \begin{subfigure}[b]{0.3\textwidth}
        \includegraphics[width=\textwidth]{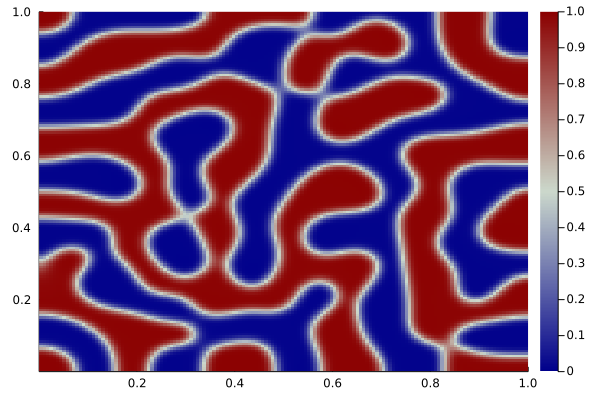}
        \caption{$u_0$}
    \end{subfigure}
    ~ 
    \begin{subfigure}[b]{0.3\textwidth}
        \includegraphics[width=\textwidth]{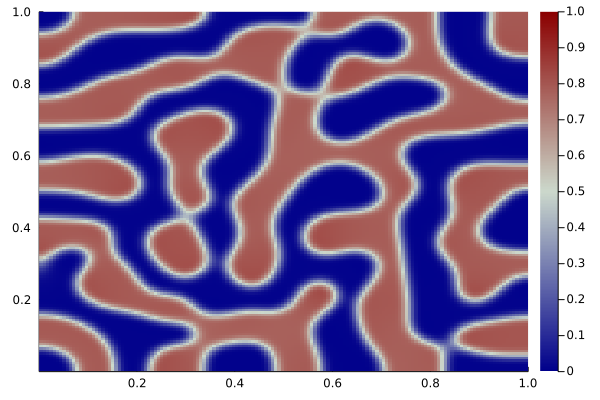} 
        \caption{$u_1$}
    \end{subfigure}
    ~ 
    \begin{subfigure}[b]{0.3\textwidth}
        \includegraphics[width=\textwidth]{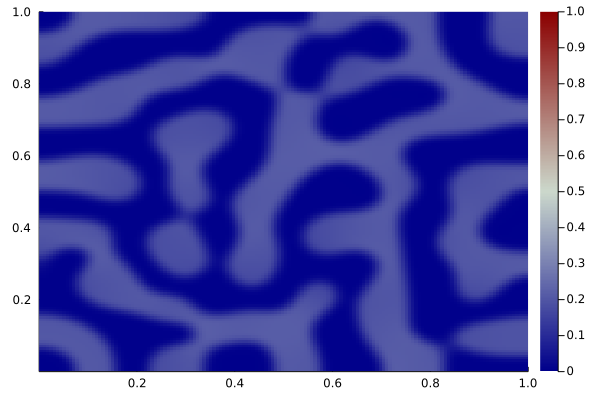}
        \caption{$u_2$}
    \end{subfigure}

    \begin{subfigure}[b]{0.3\textwidth}
        \includegraphics[width=\textwidth]{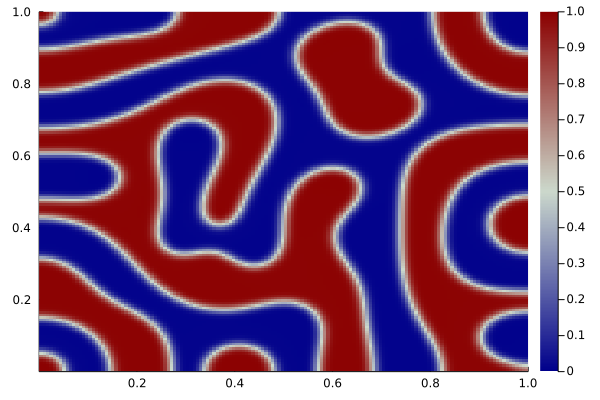}
        \caption{$u_0$}
    \end{subfigure}
    ~ 
    \begin{subfigure}[b]{0.3\textwidth}
        \includegraphics[width=\textwidth]{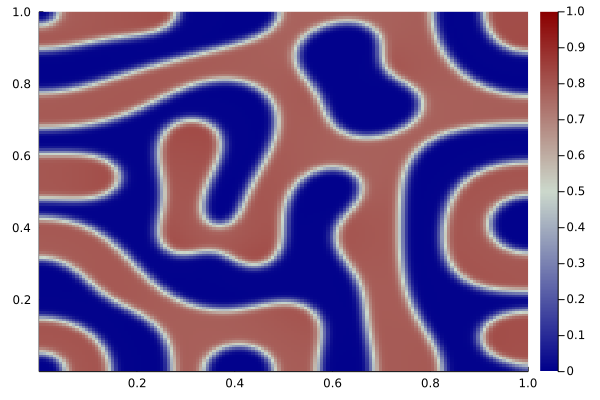} 
        \caption{$u_1$}
    \end{subfigure}
    ~ 
    \begin{subfigure}[b]{0.3\textwidth}
        \includegraphics[width=\textwidth]{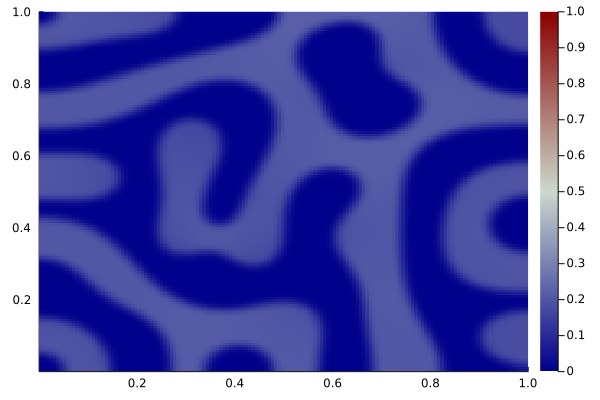}
        \caption{$u_2$}
    \end{subfigure}

    \begin{subfigure}[b]{0.3\textwidth}
        \includegraphics[width=\textwidth]{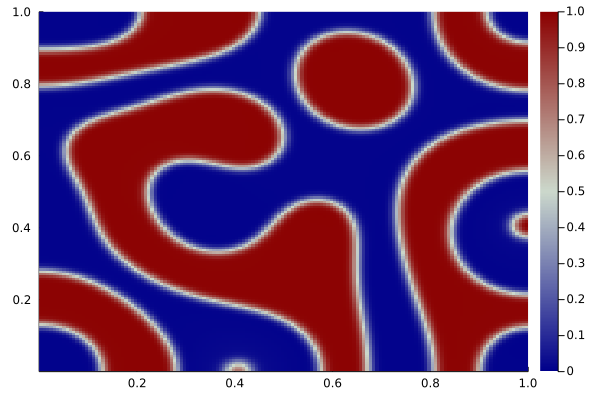}
        \caption{$u_0$}
    \end{subfigure}
    ~ 
    \begin{subfigure}[b]{0.3\textwidth}
        \includegraphics[width=\textwidth]{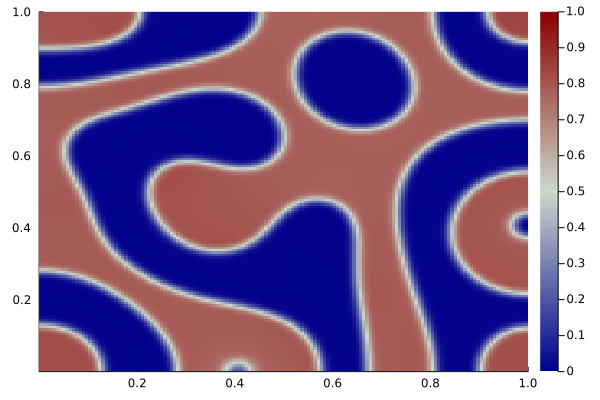} 
        \caption{$u_1$}
    \end{subfigure}
    ~ 
    \begin{subfigure}[b]{0.3\textwidth}
        \includegraphics[width=\textwidth]{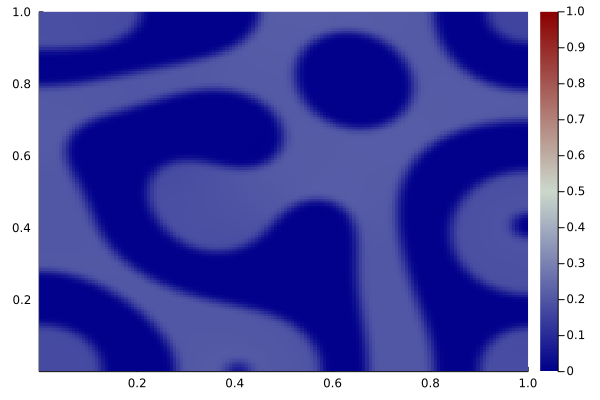}
        \caption{$u_2$}
    \end{subfigure}
    \caption{\small Spinodal decomposition successively at times $t=0, 0.06,0.13,0.49,1.5$.} 
  \label{fig:spin}
\end{figure}

\FloatBarrier

\subsection*{Acknowledgements}
The authors acknowledge support from project COMODO (ANR-19-CE46-0002). GM acknowledges the support of Deutsche Forschungsgemeinschaft (DFG) via grant GZ: MA 10100/1-1 project number 496629752. GM is member of the \emph{Gruppo Nazionale per l'Analisi Matematica, la Probabilit\`a e le loro Applicazioni}
(GNAMPA) of the Istituto Nazionale di Alta Matematica (INdAM). JFP thanks the DFG for support via the Research Unit FOR 5387 POPULAR, Project No. 461909888. 

\section{Appendix}

\begin{lemma}\label{lem:aux}
Let $\Omega\subset \mathbb{R}^d$ be a measurable bounded domain and let $u\in L^1(\Omega)$ be such that $u\geq 0$ almost everywhere on $\Omega$. Let $M:=\int_\Omega u\,dx$, $n\in \mathbb{N}^*$ and $m_1,\ldots,m_n\in \mathbb{R}_+$ such that 
$$
M:= \sum_{k=1}^n m_k.
$$
Then there exist $n$ measurable subsets $\Omega_k\subset \Omega$ for $k=1,\ldots,n$ such that
\begin{itemize}
\item $\bigcup_{k=1}^n \Omega_k= \Omega$;
\item $\Omega_k \cap \Omega_{k'} = \emptyset$ as soon as $k\neq k'$;
\item $\int_{\Omega_k} u\,dx = m_k$.
\end{itemize}
\end{lemma}

\begin{proof}
Let $x_1\in \Omega$ and consider the function $f:\mathbb{R}_+ \to \mathbb{R}_+$ defined as:
$$
\forall r\geq 0, \quad f(r):= \int_{\Omega \cap B_r} u\,dx.
$$
Then, it can be easily seen that $f$ is continuous using the Lebesgue convergence theorem, non-decreasing, such that $f(0) = 0$ and that there exists $R>0$ such that for all $r\geq R$, $f(r) = M$. This implies that there exists $r_1\geq 0$ such that $f(r_1) = m_1$, and we defined $\Omega_1 = \Omega \cap B_{r_1}$. The other sets $\Omega_2,\ldots, \Omega_n$ can be constructed using exactly the same procedure by induction on the set $\Omega \setminus \Omega_1$. 
\end{proof}

\bibliographystyle{abbrv}
\bibliography{CrossCH}

\end{document}